\documentclass[reqno]{amsart}
\usepackage{hyperref}
\usepackage{color}

\def\N{{\mathbb{N}}}

\def\R{{\mathbb{R}}}
\def\E{{\mathbb{E}}}

\def\G{{\mathbb{G}}}
\def\S{{\mathbb{S}}}
\def\W{{\mathbb{W}}}

\begin{document}
\title[Pontryagin principle and Envelope Theorem]{Pontryagin principle and Envelope Theorem}
\author[J. Blot, H. Yilmaz]
{Jo${\rm \ddot e}$l Blot, Hasan Yilmaz } 

\address{Jo\"{e}l Blot: Laboratoire SAMM EA 4543,\newline
Universit\'{e} Paris 1 Panth\'{e}on-Sorbonne, centre P.M.F.,\newline
90 rue de Tolbiac, 75634 Paris cedex 13,
France.}
\email{blot@univ-paris1.fr}
\address{Hasan Yilmaz:  Laboratoire SAMM EA 4543,\newline
Universit\'{e} Paris 1 Panth\'{e}on-Sorbonne, centre P.M.F.,\newline
90 rue de Tolbiac, 75634 Paris cedex 13,
France.}
\email{yilmaz.research@gmail.com}
\date{June, 25, 2022}
\begin{abstract}
We provide an improvement of the maximum principle of Pontryagin of the Optimal Control problems. We establish differentiability properties of the value function of problems of Optimal Control with assumptions as low as possible. Notably, we lighten the assumptions by using G\^ateaux and Hadamard differentials.
\end{abstract}
\numberwithin{equation}{section}
\newtheorem{theorem}{Theorem}[section]
\newtheorem{lemma}[theorem]{Lemma}
\newtheorem{example}[theorem]{Example}
\newtheorem{remark}[theorem]{Remark}
\newtheorem{definition}[theorem]{Definition}
\newtheorem{corollary}[theorem]{Corollary}
\newtheorem{proposition}[theorem]{Proposition}
\thispagestyle{empty} \setcounter{page}{1}
\maketitle
\vskip4mm
%
\section{Introduction}
The paper provides envelope theorems for parameterized problems of Optimal Control (problem of Bolza) as
\[
({\mathcal B},\pi)
\left\{
\begin{array}{cl}
{\rm Maximize} & \int_0^T f^0(t,x(t),u(t),\pi)dt + g^0(x(T),\pi) \\
{\rm subject \;  to} & x \in PC^1([0,T], \Omega), u \in NPC^0([0,T], U)\\
\null & \forall t\in[0,T],\, \underline{d}x(t) = f(t,x(t), u(t),\pi), \; x(0) = \xi_0\\
\null & \forall i \in \{ 1,..., m\}, \; \; g^i(x(T),\pi) \geq 0\\
\null & \forall j \in \{ 1,..., q\}, \; \; h^j(x(T),\pi) = 0.
\end{array}\right.
\]
For all parameter $\pi$, $V[\pi]$ is the value of the problem of (${\mathcal B},\ \pi)$.\\
We establish properties of the value function $[\pi\mapsto V[\pi]]$ in terms of G\^ateaux variation, G\^ateaux differentiability and Fr\'echet continuous differentiability.\\
We try to establish such results using assumptions as low as possible.\\
Envelope theorems for static optimization and Calculus of Variations in \cite{BY2} where references on economic motivations are cited.\\
To realize that we start by establishing new Pontryagin principles for the problems of Bolza and Mayer without parameter which improve the results of \cite{BY} by lightening the assumptions. Moreover, we provide new qualification conditions which are very useful to treat the question of the envelope theorems.\\
Notice that we provide conditions on the G\^ateaux differentials to obtain conditions of Lipschitz, we provide a new result on the differentiability of nonlinear functionals. Moreover, we don't do assumptions on the regularity of the multipliers and the adjoint function with respect to the parameter (as it is often the case in the literature). \\
We summarize the content of this paper as follows.\\
In Section 2, we establish the Pontryagin principles for the Optimal Control. In a first subsection, we state the Pontryagin principle for the problem of Bolza, and we provide new qualification conditions. In a second subsection, we state the Pontryagin principle for the problem of Mayer, and we give qualification conditions. In a third subsection, we prove the results on the problem of Mayer; in order to do that we use a new multiplier rules which is an improvement of a multiplier rules in \cite{BL}. In the last subsection, in order to prove the Pontryagin principle of the problem of Bolza, we transform the problem of Bolza into a problem of Mayer.\\
In Section 3, we establish envelope theorems for parameterized problems of Optimal Control. In a first subsection, we state envelope theorems. In a second subsection, we prove the first envelope theorem; in order to do that we provide new results on the differentiability of nonlinear integral functionals and we use the new Pontryagin principles and qualification conditions for the problems of Bolza without parameter. In a third subsection, by using the first envelope theorem, we prove that the value function is G\^ateaux differentiable at a point. In the last subsection, we prove the last envelope theorem by using the second envelope theorem. 
\section{Statements of the Pontryagin Principles}
\subsection{Pontryagin principle for the problem of Bolza}  
$E$ is a real Banach space, $\Omega$ is a non-empty subset of $E$, $U$ is a Hausdorff topological space, $f:[0,T] \times \Omega \times U \rightarrow E$, $f^0:[0,T] \times \Omega \times U \rightarrow \R$, $g^\alpha: \Omega \rightarrow \R$ ($0\le \alpha\le  m$) and $h^\beta : \Omega \rightarrow \R$ ($1\le \beta \le q$) are functions when $(m,q)\in \N_*\times\N_*$ where $\N_*=\N\setminus\{0\}$.\\
When $X$ is a Hausdorff topological space, $PC^0([0,T],X)$ denotes the space of the {\it piecewise continuous functions} from $[0,T]$ into $X$.\\ As in \cite{BY}, we specify that $x\in PC^0([0,T], X)$ when $x$ is continuous on $[0,T]$ or when there exists a subdivision of $[0,T]$, $0=\tau_0<\tau_1<...<\tau_k<\tau_{k+1}=T$ such that $x$ is continuous at $t$ when $t\notin\{\tau_i : 0\le i\le k+1\}$ and the right-hand limit $x(\tau_i+)$ exists in $X$ and when $i\in\{0,...,k\}$ and the left hand limit $x(\tau_i-)$ exists in $X$ when $i\in\{1,...,k+1\}$.\\
We define $NPC^0_R([0,T],X)$ as the set of the $x\in PC^0([0,T],X)$ which are right-continuous on $[0,T[$ and left-continuous at $T$. An element of $NPC^0_R([0,T],X)$ is called a {\it normalized} piecewise continuous function, cf. \cite{BY}.\\
When $X$ is a real normed vector space, $\mathfrak{O}$ is a non-empty open subset of $X$ and $Y$ is a Hausdorff space.
A mapping $\mathfrak{g}:[0,T] \times \mathfrak{O}  \rightarrow Y$ is {\it piecewise continuous with a parameter} on $[0,T]\times \mathfrak{O}$ when there exists a subdivision of $[0,T]$, $0=\tau_0<\tau_1<...<\tau_k<\tau_{k+1}=T$ such that for all $i \in \{0,...,k-1\}$, $\mathfrak{g}$ is continuous on $[\tau_i,\tau_{i+1}[\times \mathfrak{O}$, $\mathfrak{g}$ is continuous on $[\tau_k,\tau_{k+1}]\times \mathfrak{O}$, and for all $i \in \{1,...,k\}$, for all $x\in \mathfrak{O}$, $\lim\limits_{\substack{t \to \tau_i-, z \to x}}\mathfrak{g}(t,z)$ exists in $Y$. The space of all the piecewise continuous with a parameter is denoted by $PCP^0([0,T]\times \mathfrak{O},Y)$.\\
When $X$ is included into a real normed vector space, $PC^1([0,T],X)$ denotes the space of the piecewise continuously differentiable functions from $[0,T]$ into $X$. We specify that $x\in PC^1([0,T],X)$ when $x$ is continuously differentiable on $[0,T]$ or when $x$ is continuous on $[0,T]$ and there exists a subdivision , $0=\tau_0<\tau_1<...<\tau_k<\tau_{k+1}=T$ such that $x$ is continuously differentiable at $t$ if $t\notin \{\tau_i: 0\le i \le k+1\}$ the right derivative $x_R'(\tau_i)=x'(\tau_i+)$ exists when $i\in\{0,...,k\}$, the left derivative $x_L'(\tau_i)=x'(\tau_i-)$ exists when $i\in\{1,..,k+1\}$, cf. \cite{BY}.\\
As in \cite{BY}, we consider the {\it extended derivative} of $x\in PC^1([0,T],X)$ as follows: 
\begin{equation}\label{eq21}
\underline{d}x(t) :=
\left\{
\begin{array}{ccl}
x'(t) & {\rm if} & t \in [0,T] \setminus \{ \tau_i : i \in \{0, ..., k+1 \}\}\\
x'_R(t) & {\rm if} & t = \tau_i, i \in \{0,...,k \}\\
x'_L(t) & {\rm if} & t =T.
\end{array}
\right.
\end{equation}
Note that if $x\in PC^1([0,T],X)$ then $\underline{d}x \in NPC^0_R([0,T], X)$. Moreover, $\underline{d}$\, is a continuous linear operator from $PC^1([0,T],X)$ into $NPC^0_R([0,T], X)$.\\
When $X$ is included in a Banach space, the following relation holds: 
$$\text{for all } s<t \text{ in } [0,T],\, x(t)-x(s)=\int_{s}^{t} \underline{d}x(r)dr$$
where the integral is taken in the sense of Riemann as exposed in \cite{Di}. We refer to \cite{BY} for the details about these function spaces.\\
When $X$ is a real normed vector space, $\mathfrak{O}$ is a non-empty open subset of $X$, $x\in \mathfrak{O}$, $v\in X$, and $Y$ is a real normed vector space. When $\mathfrak{f}: \mathfrak{O} \rightarrow Y$ is a mapping.
When it exists $D_G^+\mathfrak{f}(x;v)$ denotes the right-directional derivative (also called the right G\^ateaux variation) of ${\mathfrak f}$ at $x$ in the direction $v$ cf. \cite{BY2} (Subsection 2.1).
When it exists $D_G{\mathfrak f}(x)$ (respectively $D_H{\mathfrak f}(x)$, respectively $D_F{\mathfrak f}(x)$) denotes the G\^ateaux (respectively Hadamard, respectively Fr\'echet) differential of ${\mathfrak f}$ at $x$.\\
Moreover, when $X$ is a finite product of $n$ real normed spaces, $X:=\prod_{i=1}^{n} X_i$, if $i\in\{1,...,n\}$, $D_{G,i}{\mathfrak f}(x)$ (respectively $D_{H,i}{\mathfrak f}(x)$, respectively $D_{F,i}{\mathfrak f}(x)$) denotes the partial G\^ateaux (respectively Hadamard, respectively Fr\'echet) differential of ${\mathfrak f}$ at $x$ with respect to the $i-th$ vector variable. If $1\le i_1\le i_2 \le i_3 \le n$, $D_{H,(i_1,i_2,i_3)}{\mathfrak f}(x)$ denotes the Hadamard diffential of the mapping $[({\bf x}_{i_1},{\bf x}_{i_2},{\bf x}_{i_3})\mapsto {\mathfrak f}(x_1,...,{\bf x}_{i_1},...,{\bf x}_{i_2},...,{\bf x}_{i_3},...,x_n)]$ at the point $(x_{i_1},x_{i_2},x_{i_3})$. 

We formulate the problem of Bolza:
\[
({\mathcal B})
\left\{
\begin{array}{cl}
{\rm Maximize} & J(x,u):=\int_0^T f^0(t,x(t),u(t))dt + g^0(x(T)) \\
{\rm subject \;  to} & x \in PC^1([0,T], \Omega), u \in NPC_R^0([0,T], U)\\
\null & \forall t\in [0,T],\, \underline{d}x(t) = f(t,x(t), u(t)), \; x(0) = \xi_0\\
\null & \forall \alpha \in\{ 1,...,m\}, \; \; g^{\alpha}(x(T)) \geq 0\\
\null & \forall \beta \in \{1,..., q\}, \; \; h^{\beta}(x(T)) = 0.
\end{array}\right.
\]
Generally the controlled dynamical system is present in (${\mathcal B}$) is formulated as follows: $x'(t)=f(t,x(t),u(t))$ when $x'(t)$ exists in \cite{BY}, we explain why the present formulation is equivalent.\\
If $f^0=0$, (${\mathcal B}$) is called a problem of Mayer and it is denoted by (${\mathcal M}$).\\
When $(x,u)$ is an admissible process for $(\mathcal{B})$ or $(\mathcal{M})$, we consider the following condition of qualification, for $i \in \{0,1 \}$. 
\[
(QC, i)
\left\{
\begin{array}{l}
{\rm If} \;\; (c_{\alpha})_{i \leq \alpha \leq m} \in \R_+^{1 - i+ m}, (d_{\beta})_{1 \leq \beta \leq q} \in \R^q {\rm \;\;\;satisfy} \\
(\forall \alpha \in \{ 1,...,m\}, \; c_{\alpha} g^{\alpha}(x(T)) = 0), {\rm and} \\
\sum_{\alpha = i}^m c_{\alpha} D_{G}g^{\alpha}(x(T)) + \sum_{\beta = 1}^q d_{\beta} D_{G}h^{\beta}(x(T)) = 0, {\rm then}\\
(\forall \alpha \in \{ i, ...,m\}, \;  c_{\alpha} = 0) \;\; {\rm and} \;\; (\forall \beta \in \{ 1, ..., q\}, \; d_{\beta} = 0).
\end{array}
\right.
\]
When $i=0$, this condition is due to Michel \cite{PMP}.\\
Now we formulate the assumptions for our theorems. Let $(x_0,u_0)$ be an admissible process of $(\mathcal{B})$ or $(\mathcal{M})$.
\vskip1mm
\noindent
{\bf Conditions on the integrand of the criterion.}
\begin{itemize}
\item[\bf{(A{\sc i}1)}] $f^0\in C^0([0,T] \times \Omega \times U,\R)$, for all $(t, \xi, \zeta) \in [0,T] \times \Omega \times U$, $D_{G,2}f^0(t,\xi, \zeta)$ exists, for all $(t,\zeta) \in [0,T]\times U$, $D_{F,2}f^0(t,x_0(t), \zeta)$ exists and $[(t,\zeta)\mapsto D_{F,2}f^0(t,x_0(t),\zeta)]\in C^0([0,T]\times U,E^*)$. 
\item[\bf{(A{\sc i}2)}] For all non-empty compact $K\subset \Omega$, for all non-empty compact $M\subset U$, $\sup_{(t,\xi,\zeta)\in [0,T]\times K\times M} \|D_{G,2}f^0(t,\xi,\zeta)\| <+\infty$.
\end{itemize}
where $C^0$ means the continuity and $E^*$ denotes the topological dual space of $E$.
\vskip1mm
\noindent
{\bf Conditions on the vector field.}
\begin{itemize}
\item[\bf{(A{\sc v}1)}] $f\in C^0([0,T] \times \Omega \times U,E)$, for all $(t, \xi, \zeta) \in [0,T] \times \Omega \times U$, $D_{G,2}f(t,\xi, \zeta)$ exists, for all $(t,\zeta) \in [0,T]\times U$, $D_{F,2}f(t,x_0(t), \zeta)$ exists and $[(t,\zeta)\mapsto D_{F,2}f(t,x_0(t),\zeta)]\in C^0([0,T]\times U,\mathcal{L}(E,E))$. 
\item[\bf{(A{\sc v}2)}] For all non-empty compact $K\subset \Omega$, for all non-empty compact $M\subset U$, $\sup_{(t,\xi,\zeta)\in [0,T]\times K\times M} \|D_{G,2}f(t,\xi,\zeta)\| <+\infty$,
\end{itemize}
where $\mathcal{L}(E,E)$ denotes the space of bounded linear mappings from $E$ into $E$.\\
{\bf Conditions on terminal constraints functions and on terminal part of the criterion.}
\begin{itemize}
\item[\bf{(A{\sc t}1)}] For all $\alpha \in \{0,...,m\}$, $g^{\alpha}$ is Hadamard differentiable at $x_0(T)$.
\item[\bf{(A{\sc t}2)}] For all $\beta \in \{1,...,q \}$, $h^{\beta}$ is continuous on a neighborhood of $x_0(T)$ and Hadamard differentiable at $x_0(T)$.
\end{itemize}
(A{\sc i}1) and (A{\sc i}2) are an improvement of condition (A3) of \cite{BY}, (A{\sc v}1) and (A{\sc v}2) are an improvement of condition (A4) of \cite{BY}, (A{\sc t}1) is an improvement of condition (A1) of \cite{BY}, and (A{\sc t}2) is an improvement of condition (A2) of \cite{BY}.\\
The Hamiltonian of (${\mathcal B}$) is the function $H_B :[0,T]\times \Omega \times U\times E^*\times\R \rightarrow \R$ defined by $H_B(t,\xi,\zeta,p, \lambda) := \lambda f^0(t,\xi,\zeta) + p \cdot f(t,\xi,\zeta)$ when $(t,\xi,\zeta,p, \lambda)\in [0,T]\times \Omega \times U\times E^*\times\R$.
\begin{theorem}\label{th21} (Pontryagin Principle for the problem of Bolza)\\
When $(x_0,u_0)$ is a solution of $(\mathcal{B})$, under (A{\sc i}1), (A{\sc i}2), (A{\sc v}1), (A{\sc v}2), (A{\sc t}1) and (A{\sc t}2), there exist multipliers $(\lambda_{\alpha})_{0 \leq \alpha \leq m} \in \R^{1 + m}$, $(\mu_{\beta})_{1 \leq \beta \leq q} \in \R^q$ and an adjoint function $p \in PC^1([0,T], E^*)$ which satisfy the following conditions.
\begin{itemize}
\item[(NN)] $((\lambda_{\alpha})_{0 \leq \alpha \leq m},(\mu_{\beta})_{1 \leq \beta \leq q})$ is non zero.
\item[(Si)] For all $\alpha \in \{0,...,m \}$, $\lambda_{\alpha} \geq 0$.
\item[(S${\ell}$)] For all $\alpha \in \{1,...,m \}$, $\lambda_{\alpha} g^{\alpha}(x_0(T)) = 0$.
\item[(TC)] $\sum_{\alpha = 0}^{m} \lambda_{\alpha} D_Hg^{\alpha}(x_0(T)) + \sum_{\beta = 1}^q \mu_{\beta} D_Hh^{\beta}(x_0(T)) = p(T)$.
\item[(AE.B)] $\underline{d}p(t) = - D_{F,2}H_B(t, x_0(t),u_0(t), p(t), \lambda_0)$ for all $t \in [0,T]$.
\item[(MP.B)] For all $t \in [0,T]$, for all $\zeta \in U$, \\
$H_B(t, x_0(t), u_0(t), p(t), \lambda_0) \geq H_B(t, x_0(t), \zeta, p(t), \lambda_0)$.
\item[(CH.B)] $\bar{H}_B := [ t \mapsto H_B(t, x_0(t), u_0(t), p(t), \lambda_0)] \in C^0([0,T], \R)$,
\end{itemize} 
\end{theorem}
\noindent
(NN) means non nullity, (Si) means sign, (S${\ell}$) means slackness, (TC) means transversality condition, (AE.B) means adjoint equation, (MP.B) means maximum principle and (CH.B) means continuity of the Hamiltonian.
\begin{corollary}\label{cor12}
In the setting and under the assumptions of Theorem \ref{th21}, if, in addition, we assume that, for all $(t,\xi, \zeta) \in [0,T] \times \Omega \times U$, the partial derivatives with respect to the first variable $\partial_1f^0(t, \xi, \zeta)$ and $\partial_1f(t, \xi, \zeta)$ exist and $\partial_1f^0$ and $\partial_1f$ are continuous on $[0,T] \times \Omega \times U$, then $\bar{H}_B \in PC^1([0,T], \R)$ and, for all $t \in [0,T]$, $\underline{d}\bar{H}_B(t) = \partial_1 H_B(t, x_0(t), u_0(t), p(t), \lambda_0)$.
\end{corollary}
We introduce other conditions.
\begin{itemize}
\item[\bf{(A{\sc v}3)}] $U$ is a subset of real normed vector space $Y$, there exists $\hat{t}\in [0,T]$ s.t. $U$ is a neighborhood of $u_0(\hat{t})$ in $Y$, $D_{G,3}f(\hat{t},x_0(\hat{t}),u_0(\hat{t}))$ exists and it is surjective
\end{itemize}
and 
\begin{itemize}
\item[{\bf(LI)}] $U$ is a subset of a real normed vector space $Y$ s.t. $U$ is a neighborhood of $u_0(T)$ in $Y$, $D_{G,3}f(T,x_0(T),u_0(T))$ exists and \\$((D_Hg^{\alpha}(x_0(T))\circ D_{G,3}f(T,x_0(T),u_0(T)))_{1\le \alpha\le m},\\(D_Hh^{\beta}(x_0(T))\circ D_{G,3}f(T,x_0(T),u_0(T)))_{1\le \beta \le q})$ is linearly free.
\end{itemize}
\begin{corollary}\label{cor13}
In the setting and under the assumptions of Theorem \ref{th21}, the following assertions hold.
\begin{itemize}
\item[(i)] Under (QC, 1) for $(x,u) = (x_0,u_0)$, we have for all $t \in [0,T]$, $(\lambda_0, p(t))\neq 0$. 
\item[(ii)] Under (QC, 1) for $(x,u) = (x_0,u_0)$ and (A{\sc v}3), we can choose $\lambda_0=1$. 
\item[(iii)] Under (LI), we can choose $\lambda_0=1$. 
\item[(iv)] Under (LI), if, in addition, we assume that $D_{G,3}f^0(T,x_0(T),u_0(T))$ exists, then $((\lambda_{\alpha})_{0 \leq \alpha \leq m}, (\mu_{\beta})_{1 \leq \beta \leq q},p)\in \R^{1 + m}\times\R^q\times PC^1([0,T], E^*)$ with $\lambda_0=1$, which satisfies the conclusions of Theorem \ref{th21}, are unique.
\end{itemize}
\end{corollary}
\subsection{Pontryagin principles for the problem of Mayer}
The Hamiltonian of (${\mathcal M}$) is the function $H_M :[0,T]\times \Omega \times U\times E^* \rightarrow \R$ defined by $H_M(t,\xi,\zeta,p) :=  p \cdot f(t,\xi,\zeta)$ when $(t,\xi,\zeta,p)\in [0,T]\times \Omega \times U\times E^*$.
\begin{theorem}\label{th22}(Pontryagin Principle for the problem of Mayer)\\
When $(x_0,u_0)$ is a solution of $(\mathcal{M})$, under (A{\sc v}1), (A{\sc v}2), (A{\sc t}1) and (A{\sc t}2), there exist multipliers $(\lambda_{\alpha})_{0 \leq \alpha \leq m} \in \R^{1 + m}$, $(\mu_{\beta})_{1 \leq \beta \leq q} \in \R^q$ and an adjoint function $p \in PC^1([0,T], E^*)$ which satisfy the following conditions.
\begin{itemize}
\item[(NN)] $((\lambda_{\alpha})_{0 \leq \alpha \leq m},(\mu_{\beta})_{1 \leq \beta \leq q})$ is non zero.
\item[(Si)] For all $\alpha \in \{0,...,m \}$, $\lambda_{\alpha} \geq 0$.
\item[(S${\ell}$)] For all $\alpha \in \{1,...,m \}$, $\lambda_{\alpha} g^{\alpha}(x_0(T)) = 0$.
\item[(TC)] $\sum_{\alpha = 0}^{m} \lambda_{\alpha} D_Hg^{\alpha}(x_0(T)) + \sum_{\beta = 1}^q \mu_{\beta} D_Hh^{\beta}(x_0(T)) = p(T)$.
\item[(AE.M)] $\underline{d}p(t) = - D_{F,2}H_M(t, x_0(t),u_0(t), p(t))$ for all $t \in [0,T]$.
\item[(MP.M)] For all $t \in [0,T]$, for all $\zeta \in U$, \\
$H_M(t, x_0(t), u_0(t), p(t)) \geq H_M(t, x_0(t), \zeta, p(t))$.
\item[(CH.M)] $\bar{H}_M := [ t \mapsto H_M(t, x_0(t), u_0(t), p(t))] \in C^0([0,T], \R)$.
\end{itemize} 
\end{theorem}
\begin{corollary}\label{cor22}
In the setting and under the assumptions of Theorem \ref{th22}, if in addition we assume that, for all $(t,\xi, \zeta) \in [0,T] \times \Omega \times U$, the partial derivatives with respect to the first variable $\partial_1f(t, \xi, \zeta)$ exist and $\partial_1f$ are continuous on $[0,T] \times \Omega \times U$, then $\bar{H}_M \in PC^1([0,T], \R)$ and, for all $t \in [0,T]$, $\underline{d}\bar{H}_M(t) = \partial_1 H_M(t, x_0(t), u_0(t), p(t))$.
\end{corollary}
\begin{corollary}\label{cor23}
In the setting and under the assumptions of Theorem \ref{th22}, the following assertions hold.
\begin{itemize}
\item[(i)] Under (QC, 1) for $(x,u) = (x_0,u_0)$, we have for all $t \in [0,T]$, $(\lambda_0, p(t))\neq 0$. 
\item[(ii)] Under (QC, 0) for $(x,u) = (x_0,u_0)$, we have for all $t \in [0,T]$, $p(t)\neq 0$. 
\item[(iii)] Under (QC, 1) for $(x,u) = (x_0,u_0)$ and (A{\sc v}3), we can choose $\lambda_0=1$. 
\item[(iv)] Under (LI), the $((\lambda_{\alpha})_{0 \leq \alpha \leq m}, (\mu_{\beta})_{1 \leq \beta \leq q},p)\in \R^{1 + m}\times\R^q\times PC^1([0,T], E^*)$ with $\lambda_0=1$, which satisfies the conclusions of Theorem \ref{th22}, are unique.
\end{itemize}
\end{corollary}
\subsection{Proofs of results for the problem of Mayer}
We consider \\$S := ((t_i,v_i))_{1 \leq i \leq N}$ where $t_i \in [0,T]$ satisfying $0 < t_1 \leq t_2 \leq ... \leq t_N < T$, where $v_i \in U$ and where $N \in \N_*$. We denote by $\S$ the set of such $S$.
\vskip1mm
\noindent 
When $S \in \S$ and $a = (a_1, ..., a_N) \in \R^N_+$, we define the following objects: $J(i) = J(i, S) := \{ j \in \{ 1,...,i-1 \} : t_j = t_i \}$, $b_i(a) = b_i(a,S) := 0$ if $J(i) = \emptyset$ and $b_i(a) = b_i(a,S)=\sum_{j \in J(i)} a_j$  if $J(i) \neq \emptyset$.
\\
We also define $I_i(a) = I_i(a,S) := [ t_i + b_i(a), t_i+  b_i(a) + a_i [$.\\ 
We define $\delta(S)=\min \{ \ t_{i+1}-t_i \ : i\in \{1,...,N-1\}, \ t_i<t_{i+1} \}$ and \\
$\|a\|_1=\sum_{i=1}^N \mid a_i \mid=\sum_{i=1}^N a_i \le \delta(S) $.\\
When $a\in B_{\|\cdot\|_1}(0,\delta(S))\cap \R^N_+$, we have $I_i(a) \subset [0,T]$ and $I_i(a) \cap I_j(a) = \emptyset$ when $i \neq j$ and we can define the needlelike variation of $u_0$:
\begin{equation}\label{eq31}
u_a(t) = u_a(t,S) := 
\left\{
\begin{array}{ccl}
v_i & {\rm if} & t \in I_i(a), 1 \leq i \leq N\\
u_0(t) & {\rm if} & t \in [0,T] \setminus \cup_{1 \leq i \leq N} I_i(a).
\end{array}
\right.
\end{equation}
It is easy to verify that 
\begin{equation} 
u_a=u_a(\cdot,S) \in NPC^0_R([0,T],U).
\end{equation}
We associate to the control function $u_a$ the non extendable solution $x_a$ of the Cauchy problem on $[0,T]$.
\begin{equation}\label{eq32}
\underline{d}x_a(t) = f(t,x_a(t), u_a(t)),\; \;  x_a(0) = \xi_0.
\end{equation}
In the sequel of this subsection, we arbitrarily fix a list $S = ((t_i, v_i))_{1 \leq i \leq N}$ in $\S$.
\vskip1mm 
\noindent
\begin{lemma}\label{lem11}
Let $X$ be a metric space, $Y$ be a non empty set, $Z$ be a real normed vector space and $\phi :X\times Y\rightarrow Z$ be a mapping. We assume that: for all non-empty compact subset $K$ of $X$, we have $\sup_{(x,y)\in K\times Y} \|\phi(x,y)\| <+\infty.$\\
Then, for all non-empty compact subset $K$ of $X$, there exists $\rho>0$ s.t. \\
$\sup_{(x,y)\in V(K,\rho)\times Y} \|\phi(x,y)\| <+\infty$,\\
where $V(K,\rho):=\{z\in X :d(z,K):=\inf_{k\in K} d(z,k)<\rho\}$.
\end{lemma}   
\begin{proof}
We proceed by contradiction, we assume that there exists a non-empty compact subset $K\subset X$ s.t. $\forall \varepsilon >0, \forall \gamma \in\R_+ \exists (x^{\varepsilon,\gamma},y^{\varepsilon,\gamma}) \in V(K,\varepsilon) \times Y, \,\\ \|\phi(x^{\varepsilon,\gamma},y^{\varepsilon,\gamma})\|>\gamma$.\\
Therefore taking $\varepsilon =\frac{1}{n}, \, \gamma =n$ with $n\in \N_*$, we obtain     
$ \forall n\in \N_*,  \exists (x_n,y_n) \in V(K,\frac{1}{n}) \times Y, \, \|\phi(x_n,y_n)\| >n$.\\
Hence, for all $n\in \N_*$, $\exists u_n \in K$ s.t. $d(x_n,u_n)=d(x_n,K)<\frac{1}{n}$.\\ 
Since $K$ is a compact, there exists $\psi :\N_* \rightarrow \N_*$  strictly increasing and $z\in K$ s.t. $\lim\limits_{\substack{ n\to +\infty}} u_{\psi(n)}=z$. Hence, we have also $\lim\limits_{\substack{ n\to +\infty}} x_{\psi(n)}=z$.\\
Since $K_0:=\{x_{\psi(n)}:n\in\N_*\}\cup\{z\}$ is compact, there exists $\gamma_0\in\R_+$ s.t., $\forall n\in \N_*$, $\|\phi(x_{\psi(n)},y_{\psi(n)})\| \le \sup_{(x,y)\in K_0\times Y} \|\phi(x,y)\| \le \gamma_0 <+\infty$.\\
Besides, we have also for all $n\in \N_*$, $\psi(n)<\|\phi(x_{\psi(n)},y_{\psi(n)})\| \le \gamma_0$, and since \\$\lim\limits_{\substack{n\to +\infty}} \psi(n)=+\infty$, we obtain the contradiction $+\infty \le \gamma_0 <+\infty$.
\end{proof}
\begin{lemma}\label{lem112}
Let $X$ and $Z$ be two real normed vector spaces, $Y$ be a non-empty set, $\G$ be a non-empty open subset of $X$, and $\phi: \mathbb{G}\times Y \rightarrow Z$ be a mapping. We assume that the two following conditions are fulfilled.
\begin{itemize}
\item[(a)] $\forall(x,y)\in \mathbb{G}\times Y$, $D_{G,1} \phi(x,y)$ exists.
\item[(b)] $\forall K\subset \mathbb{G}$, $K$ non-empty compact set, $\sup_{(x,y)\in K\times Y} \|D_{G,1} \phi(x,y)\| <+\infty.$
\end{itemize}
For each non-empty compact subset $K \subset \mathbb{G}$, there exist $\eta>0$, $\kappa>0$ s.t. for all $x\in K$, for all $x_1,x_2\in \overline{B}(x,\eta)$, for all $y\in Y$, $\|\phi(x_1,y)-\phi(x_2,y)\|\le \kappa \|x_1-x_2\|$.
\end{lemma}
\begin{proof}
Let $K\subset \mathbb{G}$, $K$ non-empty and compact. From Lemma \ref{lem11}, there exists $\rho >0$ s.t. $\kappa:=\sup_{(x,y)\in V(K,\rho)\times Y} \|D_{G,1}\phi(x,y)\| <+\infty$.\\
We set $\eta:=\frac{\rho}{2}$. Let $x\in K$ and $x_1,x_2\in \overline{B}(x,\eta)$. Since the balls are convex, we have $[x_1,x_2] \subset \overline{B}(x,\eta) \subset V(K,\rho)$. Using the mean value inequality (\cite{ATF}, Subsection 2.2.3, p. 143), we obtain, for all $y\in Y$, $\|\phi(x_1,y)-\phi(x_2,y)\|\le \sup_{\xi\in[x_1,x_2]} \|D_{G,1}\phi(\xi,y)\|\|x_1-x_2\| \le  \kappa \|x_1-x_2\|$.
\end{proof}
\begin{remark}
Note that we don't use a condition of continuity on $\phi$ in Lemma \ref{lem11}, we replace it by a condition of boundedness on the compact subsets. It is similar in Lemma \ref{lem112} of $D_{G,1}\phi$ instead of $\phi$. These lemmas permit us to replace the condition of partial differentiable continuity in (A3) and (A4) of \cite{BY} by the conditions (A{\sc i}2) and (A{\sc v}2).
\end{remark}
\begin{lemma}\label{lem113}
Let $X$ and $Y$ be metric spaces and $\phi\in C^0([0,T]\times X,Y)$. The Nemytskii operator $N_\phi:PC^0([0,T],X) \rightarrow PC^0([0,T],Y)$, defined by\\ $N_\phi(z):=[t\mapsto \phi(t,z(t))]$ when $z\in PC^0([0,T],X)$, is well defined and continuous. Moreover, $N_\phi(NPC_R^0([0,T],X))\subset NPC_R^0([0,T],Y)$.
\end{lemma}
\begin{proof}
Let $z\in PC^0([0,T],X)$; we set $w(t):=\phi(t,z(t))$ when $t\in[0,T]$. Since $\phi$ is continuous, we have, for all $t\in[0,T[$, $w(t+)=\phi(t,z(t+))$ and, for all $t\in \, ]0,T]$, $w(t-)=\phi(t,z(t-))$. Since the set of discontinuity points of $z$, discont($z$), is finite, discont($w$) is necessarily finite, and so $w\in PC^0([0,T],Y)$.\\
We denote by $\mathcal{G}(z)$ the graphic of $z$. Since $z\in PC^0([0,T],X)$, $cl({\mathcal G}(z))$ is compact and then, we can use the Heine-Schwartz lemma (\cite{Sc} p. 355, note (**)) and assert that :
$\forall \varepsilon >0,\, \exists \mathfrak{d}_\varepsilon>0, \, \forall (t,s)\in [0,T],\, \forall \xi\in X,\, |t-s|+d(z(t),\xi) \le \mathfrak{d}_\varepsilon \Rightarrow d(\phi(t,z(t)),\phi(s,\xi)) \le \varepsilon.$ Hence, we have:\\ $\forall \varepsilon >0,\, \exists \mathfrak{d}_\varepsilon>0, \, \forall t\in [0,T],\, \forall \xi\in X,\, d(z(t),\xi) \le \mathfrak{d}_\varepsilon \Rightarrow d(\phi(t,z(t)),\phi(t,\xi)) \le \varepsilon. $\\
Let $\varepsilon >0$; if $z_1\in PC^0([0,T],X)$ satisfies $\sup_{0\le t \le T} d(z(t),z_1(t))\le \mathfrak{d}_\varepsilon$, then $\sup_{0\le t\le T}d(\phi(t,z(t)),\phi(t,z_1(t)))\le \varepsilon$.\\
The continuity of $N_\phi$ at $z$ is proven. Therefore, $N_\phi$ is well defined and continuous.
Moreover, when $z\in NPC_R^0([0,T],X)$, since $z$ is right-continous on $[0,T[$ and $\phi$ is continuous, we have also $N_\phi(z) \in NPC_R^0([0,T],Y)$.  
\end{proof}
\begin{lemma}\label{lem114} Let $0=s_0<s_1<...<s_{\mathfrak n}<s_{{\mathfrak n}+1}=T$, and, for all $i\in\{0,...,{\mathfrak n}\}$, ${\bf h}_i\in PC^0([0,T],E)$. We consider ${\bf h} :[0,T] \rightarrow E$ defined by ${\bf h}(t)=\sum_{0\le i\le {\mathfrak n}-1} 1_{[s_i,s_{i+1}[}(t){\bf h}_i(t)+1_{[s_{\mathfrak n},T]}(t){\bf h}_\mathfrak{n}(t)$ when $t\in [0,T]$.\\
We have ${\bf h}\in PC^0([0,T],E)$.
\end{lemma}
\begin{proof}
Note that discont(${\bf h}$)$\subset \{s_i : 0\le i\le {\mathfrak n}+1\} \cup \bigcup_{0\le i\le {\mathfrak n}} {\rm discont}({\bf h}_i)$, hence discont(${\bf h}$) is finite. When $t\in [0,T[$, we have ${\bf h}(t+)=\sum_{0\le i\le {\mathfrak n}-1} 1_{[s_i,s_{i+1}[}(t+){\bf h}_i(t+)\\
+1_{[s_{\mathfrak n},T]}(t+){\bf h}_{\mathfrak n}(t+)$ and, when $t\in \, ]0,T]$, ${\bf h}(t-)=\sum_{0\le i\le {\mathfrak n}-1} 1_{[s_i,s_{i+1}[}(t-){\bf h}_i(t-)+1_{[s_{\mathfrak n},T]}(t-){\bf h}_{\mathfrak n}(t-)$. We have proven that ${\bf h}\in PC^0([0,T],E)$.
\end{proof}        
Now, we consider the linearization of the evolution equation
\begin{equation}\label{edoh}
{\underline d}y(t)=D_{F,2}f(t,x_0(t),u_0(t))\cdot y(t).
\end{equation}
We denote by $R(\cdot,\cdot)$ the resolvent of (\ref{edoh}). We know that, for all $(t,s)\in[0,T]^2$, $R(t,\cdot)\in PC^1([0,T],\mathcal{L}(E,E))$ and $R(\cdot,s)\in PC^1([0,T],\mathcal{L}(E,E))$.\\
For all $a\in \overline{B}_{\|\cdot\|_1}(0,\delta(S)) \cap \R^N_+$, we also consider the following Cauchy problem on an inhomogeneous ODE:
\begin{equation}\label{edoz}
\left.
\begin{array}{l}
{\underline  d}z_a(t)=D_{F,2}f(t,x_0(t),u_0(t))\cdot z_a(t)+f(t,x_0(t),u_a(t))-f(t,x_0(t),u_0(t))\\
z(0)=0.
\end{array}
\right\}
\end{equation}
We denote by $z_a$ the non extendable solution of (\ref{edoz}).
Since (\ref{edoz}) is an inhomogeneous linear ODE, $z_a$ is defined on all over $[0,T]$.
\begin{lemma}\label{zaDif}
We consider the linear mapping $\mathfrak{L}: \R^N\rightarrow E$, defined by \\
$\mathfrak{L}\cdot a:= \sum_{i=1}^N a_iR(T,t_i)\cdot [f(t_i,x_0(t_i),v_i)-f(t_i,x_0(t_i),u_0(t_i))]$ when \\
$a=(a_1,...,a_N)\in \R^N$. There exists $\varrho_1: \overline{B}_{\|\cdot\|_1}(0,\delta(S)) \cap \R^N_+\rightarrow E$ s.t. $\lim\limits_{\substack{a \to 0}} \varrho_1(a)=0$ and, for all $a\in \overline{B}_{\|\cdot\|_1}(0,\delta(S)) \cap \R^N_+$, $z_a(T)=z_0(T)+\mathfrak{L}\cdot a +\|a\|_1\varrho_1(a)$. 
\end{lemma}
\begin{proof}
For all $a\in\overline{B}_{\|\cdot\|_1}(0,\delta(S)) \cap \R^N_+$, the second member of (\ref{edoz}) is $\Delta_a$, defined by $\Delta_a(t):=0$ when $t\notin \bigcup_{1\le i\le N} I_i(a)$ and $\Delta_a(t):=f(t,x_0(t),v_i)-f(t,x_0(t),u_0(t))$ when $t\in I_i(a)$ ($1\le i \le N$). Using Lemma \ref{lem113} and \ref{lem114}, we see that $\Delta_a\in PC^0([0,T],E)$. Since $R(t,\cdot)$ is continuous, $R(t,\cdot)\cdot \Delta_a \in PC^0([0,T],E)$, hence it is Riemann integrable. \\
Using the formula of the Variation of Constants, we can write $z_a(t)=z_a(0)+ \int_{0}^{t} R(t,s)\cdot \Delta_a(s)ds$. Note that $z_0=0$ since the second member of (\ref{edoz}) is zero and the initial value is zero.\\
Hence, we have for all $a\in \overline{B}_{\|\cdot\|_1}(0,\delta(S)) \cap \R^N_+$,
$
z_a(T)-z_0(T)-\mathfrak{L}\cdot a\\
=\sum_{i=1}^{N} \int_{I_{i}(a)} R(T,s)\cdot \Delta_a(s) ds-\sum_{i=1}^{N} a_i R(T,t_i)\cdot \Delta_a(t_i)\\ 
=\sum_{i=1}^{N} \int_{I_{i}(a)} (R(T,s)\cdot \Delta_a(s) - R(T,t_i)\cdot \Delta_a(t_i)) ds.\\ 
$
We set, for all $i\in \{1,...,N\}$, $\phi_i(a):=0$ if $a_i=0$ and 
\begin{equation}
\phi_i(a):= \frac{1}{a_i}\int_{t_i+b_i}^{t_i+b_i+a_i}[R(T,s)\cdot \Delta(s,a)-R(T,t_i)\cdot \Delta(t_i)]ds.
\end{equation}
if $a_i\neq 0$. Hence, we obtain the following relation 
\begin{equation}
z_a(T)=z_0(T)+\mathfrak{L}\cdot a +\sum_{1\le i\le N} a_i \phi_i(a).
\end{equation} 
We introduce the mappings, $\psi_i:[0,1]\times (\overline{B}_{\|\cdot\|_1}(0,\delta(S)) \cap \R^N_+) \rightarrow E$, ($1\le i\le N$), defined by 
\begin{equation}
\psi_i(\theta,a):= R(T,t_i+b_i+\theta a_i)\cdot \Delta_a(t_i+b_i+\theta a_i,a)-R(T,t_i)\cdot \Delta_a(t_i).  
\end{equation}
Note that $\psi_i(\cdot,a)\in PC^0([0,1],E)$ and so it is Riemann integrable and using a change of variable, we obtain $\phi_i(a):=\int_{0}^{1}\psi_i(\theta,a)d\theta$ ($1\le i\le N$).\\
Since $x_0$ and $R(T,\cdot)$ are continuous, since $u_0$ is right-continuous, we have $\lim\limits_{\substack{a\to 0}}x_0(t_i+b_i+\theta a_i)=x_0(t_i)$, $\lim\limits_{\substack{a\to 0}}u_0(t_i+b_i+\theta a_i)=u_0(t_i)$, $\lim\limits_{\substack{a\to 0}} R(T,t_i+b_i+\theta a_i)=R(T,t_i)$, and then we obtain $\lim\limits_{\substack{a\to 0}} \psi_i(\theta,a)=0$ for all $\theta \in [0,1]$ and also 
\begin{equation}\label{210}
\lim\limits_{\substack{a\to 0}} \|\psi_i(\theta,a)\|=0 \text{ for all }\theta \in[0,1].
\end{equation}
Since $\Delta_a\in PC^0([0,T],E)$ and since $a$ is not present in the formula of $\Delta_a$, we see that:
\begin{equation}\label{2101}
\exists c_1\in \R_+,\, \forall a \in \overline{B}_{\|\cdot\|_1}(0,\delta(S)) \cap \R^N_+,\, \forall t\in[0,T],\, \|\Delta_a(t)\|\le c_1.
\end{equation}
Consequently, we obtain, for all $\theta\in [0,1]$ and for all $a\in \overline{B}_{\|\cdot\|_1}(0,\delta(S)) \cap \R^N_+$,
\begin{equation}\label{211}
\|\psi_i(\theta,a)\| \le 2\sup_{0\le s\le T} \|R(T,s)\|\|\Delta_a(s)\| \le 2\|R(T,\cdot)\|_\infty c_1.
\end{equation}
Since $\|\psi_i(\cdot,a)\|\in PC^0([0,T],\R)$, it is Riemann integrable and also Borel integrable on $[0,1]$, and since the constants are Riemann integrable on $[0,1]$, with (\ref{210}) and (\ref{211}) we can use the theorem of the Dominated Convergence of Lebesgue to obtain $ \lim\limits_{\substack{a\to 0}} \int_{0}^{1} \|\psi_i(\theta,a)\|d\theta=\int_{0}^{1} \lim\limits_{\substack{a\to 0}} \|\psi_i(\theta,a)\|d\theta=0$.
Since, for all $a\in \overline{B}_{\|\cdot\|_1}(0,\delta(S)) \cap \R^N_+$, $\|\phi_i(a)\| \le \int_{0}^{1} \|\psi_i(\theta,a)\|d\theta$, we obtain $\lim\limits_{\substack{a\to 0}} \|\phi_i(a)\|=0$, i.e. $\lim\limits_{\substack{a\to 0}} \phi_i(a)=0$.\\
Setting, for all $a\in \R^N_+\cap \overline{B}_{\|\cdot\|_1}(0,\delta(S))$, $\varrho_1(a):=0$ when $a=0$ and $\varrho_1(a):=\frac{1}{\|a\|_1} \sum_{i=1}^N a_i \phi_i(a)$ when $a\neq 0$, we see that $\lim\limits_{\substack{a\to 0}} \varrho_1(a)=0$, and we have proven the lemma.
\end{proof}
\begin{lemma}\label{lem13} There exists $k \in \R_{+*}$, for all $a \in \overline{B}_{\|\cdot\|_1}(0,\delta(S)) \cap \R^N_+$, we have $\int_0^T \| f(t, x_0(t), u_a(t)) - f(t, x_0(t), u_0(t)) \| dt \leq k \| a \|_1$.
\end{lemma}
\begin{proof}
We set $k:=c_1$ provided by ($\ref{2101})$. Let $a\in \R^N_+ \cap \overline{B}_{\|\cdot\|_1}(0,\delta(S))$; using the Chasles relation we have
\[
\begin{array}{l}
\int_{0}^{T} \|f(t,x_0(t),u_a(t))-f(t,x_0(t),u_0(t))\|dt=\int_{0}^{T} \|\Delta_a(t)\|dt \\
= \sum_{i=1}^{N}  \int_{t_i+b_i}^{t_i+b_i+a_i} \|\Delta_a(t)\|dt \le c_1\sum_{i=1}^{N}  a_i  = k\|a\|_1.
\end{array}
\]
\end{proof}
The discontinuity points of $u_0$ are denoted by $\tau_j$ ($0\le j\le k+1$). We consider the set $M := (\bigcup_{0 \leq i \leq k} cl(u_0([ \tau_i, \tau_{i+1}])) \cup \{ v_i : 1 \leq i \leq N \}$ which is compact as a finite union of compacts.
\begin{lemma}\label{lem12}
There exist $L \in \R_{+*}$ and $r \in \R_{+*}$  such that, $\forall t \in [0,T]$,
$\forall \xi, \xi_1 \in \overline{B}(x_0(t), r)$, $\forall \zeta \in M$, we have $\| f(t, \xi, \zeta) - f(t, \xi_1, \zeta) \| \leq L \| \xi - \xi_1 \|.$
\end{lemma}
\begin{proof}
We set $K:=x_0([0,T])$ which is compact and non-empty. Using (A{\sc v}1) and (A{\sc v}2) we can apply Lemma \ref{lem112} to the mapping $\phi(\xi,(t,\zeta)):=f(t,\xi,\zeta)$, with $Y:=[0,T]\times M$, to obtain the result.
\end{proof}
Setting $r_1 := r e^{-L \cdot T}$, we consider the set $\mathcal{X} := \overline{B}(x_0, r_1)\subset C^0([0,T],\Omega) \subset C^0([0,T],E)$. This last vector space is endowed with the norm of Bielecki $\| \varphi \|_b := \sup_{t \in [0,T]} (e^{- L t} \| \varphi (t) \|)$ for which it is a Banach space cf. (\cite{GD}, p.56). When $a\in \overline{B}_{\|\cdot\|_1}(0,\delta(S))\cap \R^N_+$, we consider the operator $\Phi_a:{\mathcal X} \rightarrow C^0([0,T],E)$ defined by 
\begin{equation}\label{eq38}
\Phi_a(x) := [t \mapsto \xi_0 + \int_0^t f(s, x(s), u_a(s)) ds].
\end{equation}
This operator was used in \cite{BY}.
\begin{lemma}\label{lem35} The following assertions hold.
\begin{itemize}
\item[(i)] There exists $r_2\in \R_{+*}$ s.t. for all $a \in \R^N_+$, $\Vert a \Vert_1 \leq r_2 \Rightarrow \Phi_a(\mathcal{X}) \subset \mathcal{X}$.
\item[(ii)] For all $a \in \overline{B}_{\|\cdot\|_1}(0,r_2)\cap \R^N_+$, for all $x, z \in \mathcal{X}$, \\ $\Vert \Phi_a(x) - \Phi_a(z) \Vert_b \leq (1 - e^{-L \cdot T}) \Vert x - z \Vert_b$.
\item[(iii)] For all $x \in \mathcal{X}$, the mapping $[a \mapsto \Phi_a(x)]$ is continuous from $\overline{B}_{\|\cdot\|_1}(0,r_2) \cap \R^N_+$ into $\mathcal{X}$.
\end{itemize}
\end{lemma}
\begin{proof} Using Lemma \ref{lem13} instead of Lemma 4.1 in \cite{BY} and Lemma \ref{lem12} instead of Lemma 4.2, the proof of (i) is similar to the proof of Lemma 4.3 of \cite{BY}, the proof of (ii) is similar to the proof of Lemma 4.4 of \cite{BY} and the proof of (iii) is similar to the proof of Lemma 4.5 of \cite{BY}  
\end{proof}
\begin{lemma}\label{prop38}
The following assertions hold.
\begin{itemize}
\item[(i)] For all $a \in \overline{B}_{\|\cdot\|_1}(0,r_2) \cap \R^N_+$, there exists a solution $x_a$ of the Cauchy problem (\ref{eq32}) which is defined on $[0,T]$ all over.
\item[(ii)] The mapping $[a \mapsto x_a]$, from $\overline{B}_{\|\cdot\|_1}(0,r_2) \cap \R^N_+$  into $\mathcal{X}$, is continuous.
\item[(iii)] There exists $k_1\in \R_{+*}$ such that, $\forall a\in \overline{B}_{\|\cdot\|_1}(0,r_2)\cap \R_+^N$, $\forall t\in [0,T]$, $\|x(t,a)-x_0(t)\| \le k_1\|a\|_1$.
\end{itemize}
\end{lemma}
\begin{proof}
For (i) and (ii), the proof is similar to the proof of the Proposition 4.1 in \cite{BY}, the only difference is to use of Lemma \ref{lem35} of the present paper instead of Lemmas 4.3, 4.4, 4.5 of \cite{BY}.\\
For (iii), we set $k_1=ke^{L\cdot T}$.
Let $a \in \overline{B}_{\|\cdot\|_1}(0,r_2)\cap \R_+^N$.
Since $x_a$ is a fixed point of $\Phi_a$ and $x_0$ is a fixed point of $\Phi_0$, for all $t\in[0,T]$, we have \\
$x_a(t)-x_0(t)= \xi_0 +\int_{0}^{t} f(s,x_a(s),u_a(s))ds -\xi_0 -\int_{0}^{t} f(s,x_0(s),u_0(s))ds$, which implies $\|x_a(t)-x_0(t)\| \le
\int_{0}^{t} \|f(s,x_a(s),u_a(s))-f(s,x_0(s),u_0(s))\|ds \\ 
\le \int_{0}^{t} \|f(s,x_a(s),u_a(s))-f(s,x_0(s),u_a(s))\|ds+ \\
\int_{0}^{t} \|f(s,x_0(s),u_a(s)) -f(s,x_0(s),u_0(s))\|ds.$\\
Using Lemma \ref{lem13} and \ref{lem12}, we have \\
$\|x_a(t)-x_0(t)\| \le \int_{0}^{t}(L\|x_a(s)-x_0(s)\|)ds+ k\|a\|_1$.\\
Consequently, using the lemma of Gronwall (\cite{PH}, p.24), we obtain,\\
 $\forall t\in [0,T]$, $\|x_a(t)-x_0(t)\|\le k\|a\|_1e^{\int_{0}^{T} L ds}=k_1\|a\|_1$, and so (iii) is proven.
\end{proof}  
\begin{lemma}\label{xadfc}
There exists $\varrho: \overline{B}_{\|\cdot\|_1}(0,r_2) \cap \R^N_+\rightarrow E$ s.t. $\lim\limits_{\substack{a \to 0}} \varrho(a)=0$ and s.t. for all $a\in \overline{B}_{\|\cdot\|_1}(0,r_2) \cap \R^N_+$, $x_a(T)=x_0(T)+\mathfrak{L}\cdot a +\|a\|_1\varrho(a),$
where $\mathfrak{L}$ is provided by Lemma \ref{zaDif}.
\end{lemma}
\begin{proof}
We arbitrarily fix $a\in \overline{B}_{\|\cdot\|_1}(0,r_2) \cap \R^N_+$, and we introduce, $\forall t\in [0,T]$,
\begin{equation}\label{214}
y_a(t)=(x_a(t)-z_a(t))-(x_0(t)-z_0(t))=x_a(t)-z_a(t)-x_0(t).    
\end{equation}
and 
\begin{equation}\label{215}
\gamma_a(t):={\underline d}y_a(t)-D_{F,2}f(t,x_0(t),u_0(t))\cdot y_a(t).
\end{equation}
Doing a straightforward calculation, we obtain $\underline{d}y_a(t)=f(t,x_a(t),u_a(t))-\\
D_{F,2}f(t,x_0(t),u_0(t))\cdot z_a(t)-f(t,x_0(t),u_a(t))$, and consequently $\forall t\in [0,T]$, 
\begin{equation}\label{216}
\gamma_a(t):=f(t,x_a(t),u_a(t))-f(t,x_0(t),u_a(t))-D_{F,2}f(t,x_0(t),u_0(t))\cdot (x_a(t)-x_0(t)).
\end{equation}
For all $t\in [0,T]$, we define $\varepsilon_a^1(t):=0$ if $x_a(t)=x_0(t)$, and 
\\
$\varepsilon_a^1(t):=\frac{1}{\|x_a(t)-x_0(t)\|} (f(t,x_a(t),u_a(t))-f(t,x_0(t),u_a(t))-D_{F,2}f(t,x_0(t),u_a(t))\cdot (x_a(t)-x_0(t)))$ if $x_a(t)\neq x_0(t)$. We also define $\varepsilon_a^2(t)=D_{F,2}f(t,x_0(t),u_a(t))-D_{F,2}f(t,x_0(t),u_0(t))$.\\
Doing a straightforward calculation we obtain 
\begin{equation}\label{217}
\gamma_a(t)=\|x_a(t)-x_0(t)\|\varepsilon_a^1(t)+\varepsilon_a^2(t)\cdot (x_a(t)-x_0(t)).
\end{equation}
Now, we study the properties of $\varepsilon_a^1$. Let $a\in \overline{B}_{\|\cdot\|_1}(0,r_2)\cap \R_+^N$. Let $t_0 \in[0,T]$ s.t. $x_a(t_0)=x_0(t_0)$. Since $x_a$ and $x_0$ are continuous, there exists $\nu >0$ s.t. $x_a(t)\neq x_0(t)$ when $t\in \,]t_0-\nu,t_0+\nu[$. Using the continuity of $x_a$ and $x_0$, the piecewise continuity of $u_a$ and $u_0$, the continuity of $f$, Lemma \ref{lem113} and (A{\sc v}1), we obtain that $\varepsilon_a^1(t_0+)$ and $\varepsilon_a^1(t_0-)$ exit in $E$.
When $x_a(t_0)\neq x_0(t_0)$, from the existence of $D_{F,2}f(t,x_0(t),u_a(t))$ we have 
\begin{equation}\label{218}
\left.
\begin{array}{l}
\forall \epsilon >0,\, \exists \mathfrak{d}_{\epsilon,a}>0,\, \forall \xi \in E,\|\xi-x_0(t)\|\le \mathfrak{d}_{\epsilon,a}\Rightarrow \\
\|f(t,\xi,u_a(t))-f(t,x_0(t),u_a(t))\\-D_{F,2}f(t,x_0(t),u_a(t))\cdot (\xi-x_0(t))\|\le \epsilon \|\xi-x_0(t)\|.
\end{array}
\right\}
\end{equation}
Since $\lim\limits_{\substack{t\to t_0}}(x_a(t)-x_0(t))=x_a(t_0)-x_0(t_0)=0$, when we fix $\epsilon>0$, there exists $\mathfrak{b}_{\epsilon,a}>0$ s.t. $t_0<t<t_0+\mathfrak{b}_{\epsilon,a} \Rightarrow \|x_a(t)-x_0(t)\|\le \mathfrak{d}_{\epsilon,a}$
$\Rightarrow \|f(t,x_a(t),u_a(t))-f(t,x_0(t),u_a(t))-D_{F,2}f(t,x_0(t),u_a(t))\cdot (x_a(t)-x_0(t))\| \le \epsilon\|x_a(t)-x_0(t)\|$ thanks to (\ref{218}). Therefore $\|\varepsilon_a^1(t)\| \le \epsilon$ if $x_a(t) \neq x_0(t)$ or if $x_a(t)=x_0(t)$. Hence, we have proven that $\|\varepsilon_a^1(t_0+)\|=0$. Similarly, we obtain $\|\varepsilon_a^1(t_0-)\|=0$. Consequently, we have proven
 \begin{equation}\label{219}
\|\varepsilon_a^1\|\in Reg([0,T],\R)  
\end{equation}
where $Reg([0,T],\R)$ denotes the space of the regulated functions from $[0,T]$ into $\R$ cf. \cite{Di} (Chapter 7, Section 6). Hence, $\|\varepsilon_a^1\|$ is Riemann integrable on $[0,T]$ and also Borel integrable on $[0,T]$.\\
From (A{\sc v}2), we know that $L_1:=\sup_{\zeta \in M} \|D_{F,2}f(t,x_0(t),\zeta)\| <+\infty$, and using Lemma \ref{lem12}, we obtain: $\|\varepsilon_a^1(t)\| \le \max\{0,L\}+L_1=:L_2$, and so we have 
\begin{equation}\label{220}
\exists L_2\in\R_{+*},\, \forall a \in \overline{B}_{\|\cdot\|_1}(0,r_2)\cap \R_+^N,\, \forall t\in[0,T],\, \|\varepsilon_a^1(t)\|\le L_2.
\end{equation}
We introduce the mapping $\Theta: \Omega \times [0,T]\times U \rightarrow E$ defined by 
\\$\Theta(\xi,t,\zeta):=\frac{1}{\|\xi-x_0(t)\|} (f(t,\xi,\zeta)-f(t,x_0(t),\zeta)-D_{F,2}f(t,x_0(t),\zeta)\cdot (\xi-x_0(t)))$ when $\xi\neq x_0(t)$ and $\Theta(\xi,t,\zeta)=0$ when $\xi=x_0(t)$.
We fix $(t,\zeta)\in [0,T]\times U$. From (A{\sc v}1), for all $\epsilon >0$, there exists $\mathfrak{d}_\epsilon >0$ s.t. $\|\xi-x_0(t)\|\le \mathfrak{d}_\epsilon \Rightarrow \|f(t,\xi,\zeta)-f(t,x_0(t),\zeta)-D_{F,2}f(t,x_0(t),\zeta)\cdot (\xi-x_0(t))\| \le \epsilon \|\xi-x_0(t)\|$ which implies    
\begin{equation}\label{221}
\forall (t,\zeta)\in [0,T]\times U,\,     \lim\limits_{\substack{\xi \to x_0(t)}} \Theta(\xi,t,\zeta)=0. 
\end{equation}
We fix $t\in [0,T]$, for all $a\in \overline{B}_{\|\cdot\|_1}(0,r_2)\cap \R_+^N$, we have 
\[
\begin{array}{l}
\|\varepsilon_1(t,a)\|=\|\Theta(x_a(t),t,u_a(t))\| = \|1_{[0,t_1[}(t) \Theta(x_a(t),t,u_0(t))+\\ \sum_{i=1}^N 1_{I_i(a)}(t) \Theta(x_a(t),t,v_i)+ \sum_{i=1}^{N -1} 1_{[t_i + b_i(a) + a_i, t_{i+1} + b_{i+1}(a)[}(t) \Theta(x_a(t),t,u_0(t))\\+ 1_{[t_N + b_N(a) + a_N, T]}(t) \Theta(x_a(t),t,u_0(t))\|.\\
\le (N+1)\|\Theta(x_a(t),t,u_0(t))\|+ \sum_{i=1}^N \|\Theta(x_a(t),t,v_i)\|
\end{array}
\]
and using (\ref{221}), we obtain 
\begin{equation}\label{222}
\forall t\in[0,T], \, \lim\limits_{\substack{a\to 0}}\|\varepsilon_a^1(t)\|=0.
\end{equation}
From (\ref{219}), (\ref{220}) and (\ref{222}), since the constants are Lebesgue integrable, using the Dominated Convergence Theorem of Lebesgue we obtain 
\begin{equation}\label{223}
\lim\limits_{\substack{a\to 0}}\int_{0}^{T} \|\varepsilon_a^1(t)\|\,dt =\int_{0}^{T} \lim\limits_{\substack{a \to 0}}\|\varepsilon_a^1(t)\|\, dt=0.
\end{equation}
Using (A{\sc v}1) and Lemma \ref{lem113}, we see that $\varepsilon_a^2$ is a difference of two piecewise continuous functions on $[0,T]$, and consequently we have
\begin{equation}\label{224}
\text{for all } a \in \overline{B}_{\|\cdot\|_1}(0,r_2)\cap \R_+^N, \, \|\varepsilon_a^2\| \in PC^0([0,T],\R)
\end{equation}
hence $\|\varepsilon_a^2\|$ is Riemann integrable and Lebesgue integrable on $[0,T]$.
Besides, we have also $\int_{0}^{T}\|\varepsilon_a^2(t)\|dt \le \sum_{1\le i\le N} \int_{t_i+b_i}^{t_i+b_i+a_i} (2L_1)dt  =2L_1\|a\|_1$, and so we obtain 
\begin{equation}\label{225}
\lim\limits_{\substack{a \to 0}} \int_{0}^{T}\|\varepsilon_a^2(t)\|dt=0.
\end{equation}
From (\ref{217}), (\ref{223}) and (\ref{225}), we have        
$\|\gamma_a(t)\|  \le \|x_a(t)-x_0(t)\|\|\varepsilon_1(t,a)\|+\|\varepsilon_2(t,a)\|\|x_a(t)-x_0(t)\|
\le k_1\|a\|_1(\|\varepsilon_a^1(t)\|+\|\varepsilon_a^2(t)\|)\\
\Rightarrow \int_{0}^{T}\|\gamma_a(t)\|dt \le k_1\|a\|_1( \int_{0}^{T} \|\varepsilon_a^1(t)\|dt +\int_{0}^{T}\|\varepsilon_a^2(t)\|dt)$.
Consequently, using (\ref{223}) and (\ref{225}), we have 
\begin{equation}\label{226}
\lim\limits_{\substack{a\to 0}} \left(\frac{1}{\|a\|}\int_{0}^{T}\|\gamma_a(t)\|dt\right)=0.
\end{equation}
From (\ref{215}) and the formula of the Variation of Constants, we obtain $y_a(T)=\int_{0}^{T} R(T,s)\cdot \gamma_a(s)ds$. We introduce $\varpi(a):=0$ when $a=0$ and \\
$\varpi(a):=\frac{1}{\|a\|} \int_{0}^{T} R(T,s)\cdot \gamma_a(s)ds$ when $a\neq 0$; hence we have $y_a(T)=\|a\|_1\varpi(a)$.\\
Since $R(T,\cdot)$ is piecewise continuous, it is bounded. We set $\mathfrak{q}:=\sup_{0\le s\le T}\|R(T,s)\|$. We have $\|\varpi(a)\|\le \mathfrak{q}.\frac{1}{\|a\|} \int_{0}^{T} \|\gamma_a(s)\|ds$ when $a\neq 0$, and using (\ref{226}), we obtain $\lim\limits_{\substack{a \to 0}}\|\varpi(a)\|=0$, i.e. $\lim\limits_{\substack{a \to 0}} \varpi(a)=0$. Using (\ref{214}) and Lemma \ref{zaDif}, we obtain that \\
$x_a(T)=x_0(T)+z_a(T)+y_a(T)=x_0(T)+\mathfrak{L}\cdot a+\|a\|_1(\varrho_1(a)+\varpi(a))$. Setting $\varrho(a):=\varrho_1(a)+\varpi(a)$ we have $\lim\limits_{\substack{a\to 0}}\varrho(a)=0$, and the lemma is proven.
\end{proof}
\begin{lemma}\label{lem222} Let $S = ((t_i,v_i))_{1 \leq i \leq N} \in \S$. There exist $(\lambda^S_{\alpha})_{0 \leq \alpha \leq m} \in \R^{1 + m}$ and $(\mu^S_{\beta})_{1 \leq \beta \leq q} \in \R^q$ which satisfy the following conditions.
\begin{itemize}
\item[(a)] $(\lambda^S_{\alpha})_{0 \leq \alpha \leq m}$ and $(\mu^S_{\beta})_{1 \leq \beta \leq q}$ are not simulteanous equal to zero.
\item[(b)] $\forall \alpha \in \{ 0, ..., m\}$, \; \; $\lambda^S_{\alpha} \geq 0$.
\item[(c)] $\forall \alpha \in \{1,...,m\}$, \; \; $\lambda^S_{\alpha}g^{\alpha}(x_0(T)) = 0$.
\item[(d)] $\forall i \in \{ 1,...,N\}$, $p(t_i) [f(t_i, x_0(t_i), v_i) - f(t_i, x_0(t_i), u_0(t_i))] \leq 0$, where\\
$p(t) := (\sum_{\alpha = 0}^m \lambda^S_{\alpha} D_H g^{\alpha}(x_0(T)) + \sum_{\beta = 1}^q \mu^S_{\beta} D_Hh^{\beta}(x_0(T))) R(T,t)$, $R(t,s)$ being defined just before Lemma \ref{zaDif}.
\end{itemize}
\end{lemma}
\begin{proof}
Using Lemma \ref{xadfc}, the Proposition 4.2 of \cite{BY} ensures the existence of $r_3\in \, ]0,r_2]$ and  a function $\tilde{x} \in C^0(\overline{B}_{\|\cdot\|_1}(0,r_3), \Omega)$ which is Fr\'echet differentiable at $a = 0$ and which satisfies, for all  $a \in  \overline{B}_{\|\cdot\|_1}(0,r_3) \cap \R^N_+$, $\tilde{x}(a) = x_a(T)$, and $D_F \tilde{x} (0) = \mathfrak{L}\cdot a$.
Since $(x_0,u_0)$ is a solution of $(\mathcal{M})$, $a=0$ is a solution of the following finite-dimensional optimization problem     
\[
(\mathcal{F}_S^1) :=
\left\{
\begin{array}{cl}
{\rm Maximize} & g^0(\tilde{x}(a)) \\
{\rm subject} \; \; {\rm to} & a \in B(0, r_3) \\
\null & \forall \alpha \in \{ 1, ..., m\}, \; \; g^{\alpha} ( \tilde{x} (a)) \geq 0 \\
\null & \forall \beta \in \{ 1, ..., q\}, \; \; h^{\beta}( \tilde{x} (a)) = 0 \\
\null & \forall i \in \{ 1, ..., N\}, \; \; b^*_i a \geq 0
\end{array}
\right.
\]
where $(b^*_i)_{1 \leq i \leq N}$ is the dual basis of the canonical basis of $\R^N$.\\
Since $\tilde{x}$ is Fr\'echet differentiable at $0$, by using (A{\sc t}1) and (A{\sc t}2), we have, $g^\alpha \circ \tilde{x}$ when $\alpha \in \{0,...,m\}$ and $h^\beta\circ \tilde{x}$ when $\beta\in\{1,...,q\}$ are Hadamard differentiable at $0$. Moreover, since $\tilde{x} \in C^0(\overline{B}_{\|\cdot\|_1}(0,r_3), \Omega)$, by using (A{\sc t}2), for all $\beta\in\{1,...,q\}$, $h^\beta\circ \tilde{x}$ is continuous on a neighborhood of 0. Hence we can use the Multiplier rule of \cite{HY} (Theorem 2.2) to obtain our result.
\end{proof}
With respect to Lemma 5.1 of \cite{BY}, in Lemma \ref{lem222}, the Hadamard differentiability replaces the Fr\'echet differentiability.\\
To finish the proof of Theorem \ref{th22}, we exactly proceed as in Subsection 5.2 of \cite{BY}. We just recall the schedule of the reasoning. For all $S\in \S$, we consider $K(S)$ which is the set of the $((\lambda_\alpha)_{0\le \alpha\le m},(\mu_\beta)_{1\le \beta \le q})\in \R^{1+m+q}$ which satisfy the conclusions (a,b,c,d) of Lemma \ref{lem222} and $\sum_{0\le \alpha \le m} |\lambda_\alpha|+\sum_{1\le \beta\le q} |\mu_\beta|=1$. $\Sigma(0,1)$ being the unit sphere of $\R^{1+m+q}$, $K(S)$ is a non-empty closed subset of $\Sigma(0,1)$. Since $\Sigma(0,1)$, $(K(S))_{S\in \S}$ possesses the finite intersection property (\cite{LA}, p.31) and consequently we have $\bigcap_{S\in\S} K(S) \neq \emptyset$. An element of this intersection is convenient for the conclusions (NN), (Si), (S$\ell$), (AE.M) and (MP.M) of Theorem \ref{th22}. The conclusions (CH.M) is proven by Lemma 5.2 of \cite{BY}.\\
The proof of Corollary \ref{cor22} is similar to the proof of Part (II) of Theorem 2.2 in \cite{BY} which is given in Subsection 5.3 in \cite{BY}.\\
To prove assertion (i) of Corollary \ref{cor23}, we proceed by contradiction, we assume the existence of $s\in[0,T]$ s.t. $(\lambda_0,p(s))=(0,0)$. Since $p(s)=0$, we have $p(T)=0$ since (AE.M) is linear homogeneous, and, from (TC), we have $\sum_{\alpha = 1}^m \lambda_{\alpha} D_H g^{\alpha}(x_0(T)) + \sum_{\beta = 1}^q \mu_{\beta} D_Hh^{\beta}(x_0(T))=0$. Hence using (TC), (Si), (S${\ell}$), (QC, 1) implies that $(\forall \alpha \in\{ 1,...,m\}, \lambda_{\alpha} = 0)$ and $(\forall \beta \in\{ 1,...,q\}, \mu_{\beta} = 0)$. Moreover, since $\lambda_0 = 0$, we obtain a contradiction with (NN).\\
The proof of assertion (ii) of Corollary \ref{cor23} is similar to the proof of Part (III) of Theorem 2.2 in \cite{BY}.\\
To prove assertion (iii) of Corollary \ref{cor23}, we proceed by contradiction; we assume that $\lambda_0=0$. Since $D_{G,3}f(\hat{t},x_0(\hat{t}), u_0(\hat{t}))$ exists, $D_{G,3}H_{M}(\hat{t},x_0(\hat{t}),u_0(\hat{t}),p(\hat{t}))$ exists and $D_{G,3}H_{M}(\hat{t},x_0(\hat{t}),u_0(\hat{t}),p(\hat{t}))=p(\hat{t})\circ D_{G,3}f(\hat{t},x_0(\hat{t}),u_0(\hat{t}))$.
Therefore, by using (MP.M), we have $p(\hat{t})\circ D_{G,3}f(\hat{t},x_0(\hat{t}),u_0(\hat{t}))=0$, and since $D_{G,3}f(\hat{t},x_0(\hat{t}),u_0(\hat{t}))$ is surjective, we have $p(\hat{t})=0$, hence $(\lambda_0,p(\hat{t}))=0$ that is a contradiction with the assertion (i).
We have proven that $\lambda_0 \neq 0$, and it suffices to divide all the multipliers and $p$ by $\lambda_0$ to obtain the assertion (iii).\\
To prove the assertion (iv) of Corollary \ref{cor23}, we begin to prove that $\lambda_0 \neq 0$. To do that, we proceed by contradiction, we assume that $\lambda_0=0$.\\
 Since $D_{G,3}f(T,x_0(T), u_0(T))$ exists, $D_{G,3}H_{M}(T,x_0(T),u_0(T),p(T))$ exists and we have $D_{G,3}H_{M}(T,x_0(T),u_0(T),p(T))=p(T)\circ D_{G,3}f(T,x_0(T),u_0(T))$.\\
From (MP.M), we obtain $p(T)\circ D_{G,3}f(T,x_0(T),u_0(T))=0$.\\
From (TC), we obtain:
$
\sum_{\alpha=1}^{m} \lambda_\alpha D_Hg^\alpha(x_0(T))\circ D_{G,3}f(T,x_0(T),u_0(T)) \\+
\sum_{\beta=1}^{q} \mu_\beta D_Hh^\beta(x_0(T))\circ D_{G,3}f(T,x_0(T),u_0(T))
=p(T)\circ D_{G,3}f(T,x_0(T),u_0(T))=0$, and by using (LI), we obtain $((\lambda_\alpha)_{1\le \alpha\le m},(\mu_\beta)_{1\le \beta \le q})=0$, which is a contradiction with (NN). Hence, we have proven that $\lambda_0 \neq 0$. Dividing $\lambda_\alpha, \mu_\beta,p$ by $\lambda_0$, we normalize all these terms, and we have $\lambda_0=1$.
To prove the uniqueness, let $((\lambda^1_\alpha)_{0\le \alpha\le m},(\mu^1_\beta)_{1\le \beta \le q},p^1)\in \R^{1+m+q}\times PC^1([0,T],E^*)$ which satisfy the conclusions of the Theorem \ref{th22} are verified with $\lambda_0^1=1$.\\
From (MP.M), we have, $p^1(T)\circ D_{G,3}f(T,x_0(T),u_0(T))=0$, therefore, we have $(p(T)-p^1(T))\circ D_{G,3}f(T,x_0(T),u_0(T))=0$ and from (TC), we obtain   
$
\sum_{\alpha=1}^{m} (\lambda_\alpha-\lambda^1_\alpha) D_Hg^\alpha(x_0(T))\circ D_{G,3}f(T,x_0(T),u_0(T)) \\+
\sum_{\beta=1}^{q} (\mu_\beta-\mu^1_\beta) D_Hh^\beta(x_0(T))\circ D_{G,3}f(T,x_0(T),u_0(T))=0.
$ 
Hence, using (LI), $(\lambda_\alpha)_{0\le \alpha \le m}=(\lambda^1_\alpha)_{0\le \alpha \le m}$ \text{ and }$(\mu_\beta)_{1\le \beta\le q}=(\mu^1_\beta)_{1\le \beta\le q}$, and consequently, using (TC) we obtain $p(T)=p^1(T)$. Using the uniqueness of the solution of a Cauchy problem on (AE.M), we obtain $p=p^1$. Corollary \ref{cor23} is proven.  
\subsection{Proof of the results of the problem of Bolza}
As in \cite{BY}, we transform the problem of Bolza into a problem of Mayer to deduce Theorem \ref{th21} from Theorem \ref{th22}. That is why, we introduce an additional state variable denoted by $\sigma$. We set ${\mathfrak X} := (\sigma, \xi) \in \R \times \Omega$ as a new state variable; we set $F(t, (\sigma,\xi), \zeta) := (f^0(t,\xi,\zeta), f(t,\xi,\zeta))$ as the new vectorfield; we set $G^0(\sigma,\xi) := \sigma + g^0(\xi)$, $G^{\alpha}(\sigma,\xi) := g^{\alpha}(\xi)$ when $\alpha \in \{ 1, ...,m\}$, and we set $H^{\beta}(\sigma, \xi) := h^{\beta}(\xi)$ when $\beta \in \{ 1, ..., q\}$. We formulate the new following problem of Mayer:
\[ (\mathcal{MB}) 
\left\{
\begin{array}{cl}
{\rm Maximize} & G^0({\mathfrak X}(T))\\
{\rm subject} \; \; {\rm to} & {\mathfrak X} \in PC^1([0,T], \R \times \Omega), u \in NPC_R^0([0,T], U)\\
\null & \underline{d}{\mathfrak X}(t) = F(t, {\mathfrak X}(t), u(t)), \; {\mathfrak X}(0) = (0, \xi_0)\\
\null & \forall \alpha \in \{ 1, ..., m\}, \; \; G^{\alpha}({\mathfrak X}(T)) \geq 0 \\
\null & \forall \beta \in \{ 1,..., q\}, \; \; H^{\beta}({\mathfrak X}(T)) = 0.
\end{array}
\right.
\]
Proceeding as in the section 6 of \cite{BY}, the proofs of Theorem \ref{th21}, of Corollary \ref{cor12} and of assertion (i) of Corollary \ref{cor13} are similar to the proof of Theorem \ref{th21}, Part (I), Part (II), and Part (III) of \cite{BY}.\\
\underline{Proof of assertion (ii) of Corollary \ref{cor13}.}\\ 
First we want to prove that $\lambda_0 \neq 0$. To do that, we proceed by contradiction, we assume that $\lambda_0=0$. Using (MP.B), we obtain $p(\hat{t})\circ D_{G,3}f(\hat{t},x_0(\hat{t}),u_0(\hat{t}))=0.$
Since $D_{G,3}f(\hat{t},x_0(\hat{t}),u_0(\hat{t}))$ is onto, we have necessarily $p(\hat{t})=0$. Since we have assumed (QC,1), $(\lambda_0,p(\hat{t}))=(0,0)$ provides a contradiction after assertion (i) of Corollary \ref{cor13}. 
Hence we have proven that $\lambda_0 \neq 0$. To conclude it sufficies to divide $\lambda_0,...,\lambda_m, \mu_1,...,\mu_q$ and $p$ by $\lambda_0$.\\
\underline{Proof of assertion (iii) of Corollary \ref{cor13}.}\\
First we want to prove that $\lambda_0 \neq 0$. To do that, we proceed by contradiction, we assume that $\lambda_0=0$. Using (MP.B), we obtain $p(T)\circ D_{G,3}f(T,x_0(T),u_0(T))=0$. That is why, using (TC), we obtain $\sum_{\alpha=1}^{m} \lambda_\alpha D_Hg^\alpha(x_0(T))\circ D_{G,3}f(T,x_0(T),u_0(T)) \\+
\sum_{\beta=1}^{q} \mu_\beta D_Hh^\beta(x_0(T))\circ D_{G,3}f(T,x_0(T),u_0(T))=p(T)\circ D_{G,3}f(T,x_0(T),u_0(T))\\=0.$
From (LI), we obtain ($\forall \alpha\in \{1,...,m\}$, $\lambda_\alpha=0$) and ($\forall \beta \in\{1,...,q\}$, $\mu_\beta=0$); hence we have $((\lambda_\alpha)_{0\le \alpha\le m},(\mu_\beta)_{1\le \beta \le q})=(0,0)$; which contradicts (NN). We have proven that $\lambda_0 \neq 0$. We conclude as in the proof of (ii).\\
\underline{Proof of assertion (iv) of Corollary \ref{cor13}.}\\ 
From assertion (iii), we know that there exists $((\lambda_\alpha)_{0\le \alpha\le m},(\mu_\beta)_{1\le \beta \le q})$ with $\lambda_0=1$, and $p$ which satisfy the conclusions of Theorem \ref{th21}. \\
Let $((\lambda^1_\alpha)_{0\le \alpha\le m},(\mu^1_\beta)_{1\le \beta \le q})$ with $\lambda^1_0=1$, and $p^1$ which satisfy the conclusions of Theorem \ref{th21}. \\
Using (MP.B), we have $p^1(T)\circ D_{G,3}f(T,x_0(T),u_0(T))+D_{G,3}f^0(T,x_0(T),u_0(T))=0$ and $p(T)\circ D_{G,3}f(T,x_0(T),u_0(T))+D_{G,3}f^0(T,x_0(T),u_0(T))=0$. Using twice (TC), we obtain $\sum_{\alpha=1}^{m} (\lambda_\alpha-\lambda^1_\alpha) D_Hg^\alpha(x_0(T))\circ D_{G,3}f(T,x_0(T),u_0(T)) \\+
\sum_{\beta=1}^{q} (\mu_\beta-\mu^1_\beta) D_Hh^\beta(x_0(T))\circ D_{G,3}f(T,x_0(T),u_0(T))=0$.
The linear independence provided by (LI) implies $(\forall \alpha\in \{1,...,m\},\, \lambda_\alpha-\lambda^1_\alpha=0)$ and $(\forall \beta\in \{1,...,q\}, \, \mu_\beta-\mu^1_\beta=0)$.\\
Consequently, we have $((\lambda_\alpha)_{0\le \alpha\le m},(\mu_\beta)_{1\le \beta \le q})=((\lambda^1_\alpha)_{0\le \alpha\le m},(\mu^1_\beta)_{1\le \beta \le q})$. Using twice (TC), we obtain $p(T)=p^1(T)$. Using (AE.B) and the uniqueness of the solution of a Cauchy problem, we obtain $p=p^1$. We have proven the uniqueness. 
\section{Envelope theorems}
$X$ is a Banach space, $Y$ and $Z$ are real normed spaces, $\Omega$ is a non-empty open subset of $X$, and $U$ is a non-empty open subset of $Y$, $f^0:[0,T]\times \Omega \times U \times Z \rightarrow \R$, $f:[0,T]\times \Omega \times U \times Z \rightarrow X$ $g^i: \Omega \times Z \rightarrow \R$ ($0\le i \le m$) and $h^j: \Omega \times Z \rightarrow \R$ ($1\le j \le q)$ are mappings. Let $\xi_0\in \Omega$, 
for all $\pi \in Z$, we consider the following problem of Bolza.\\
\[
({\mathcal B},\pi)
\left\{
\begin{array}{cl}
{\rm Maximize} & \int_0^T f^0(t,x(t),u(t),\pi)dt + g^0(x(T),\pi) \\
{\rm subject \;  to} & x \in PC^1([0,T], \Omega), u \in NPC_R^0([0,T], U)\\
\null & \forall t\in[0,T],\, \underline{d}x(t) = f(t,x(t), u(t),\pi), \; x(0) = \xi_0\\
\null & \forall i \in \{ 1,..., m\}, \; \; g^i(x(T),\pi) \geq 0\\
\null & \forall j \in \{ 1,..., q\}, \; \; h^j(x(T),\pi) = 0.
\end{array}\right.
\]
The Hamiltonian of Pontryagin of this problem of Bolza is $H_\pi :[0,T]\times \Omega\times U\times X^*\times \R \rightarrow \R$, defined by $H_{\pi}(t,\xi,\zeta,p, \lambda) :=  p \cdot f(t,\xi,\zeta,\pi) + \lambda f^0(t,\xi,\zeta,\pi) $ when $t\in[0,T], \xi\in \Omega$, $\zeta\in U$, $p\in X^*$ and $\lambda\in \R$. For each $\pi\in Z$, we denote by $V[\pi]$ the value of $(\mathcal{B},$ $\pi)$.
\subsection{Main results}
We fix $\pi_0\in Z$ and we consider the following list of conditions.\\
{\bf Conditions on the solutions.} 
\begin{itemize}
\item[{\bf (SO)}] There exists an open neighborhood $P$ of $\pi_0$ in $Z$ s.t., $\forall \pi \in P,$ there exists $(x[\pi],u[\pi]) \in PC^1([0,T],\Omega)\times NPC_{R}^0([0,T],U)$, a solution of $({\mathcal B},\pi)$. There exists $\delta\pi \in Z$ s.t. $D_G^+x[\pi_0;\delta\pi]$ and $D_G^+u[\pi_0;\delta\pi]$ exist.  
\end{itemize}
{\bf Conditions on the integrand of the criterion.}
\begin{itemize}
\item[{\bf (IC1)}]  $f^0 \in C^0([0,T] \times \Omega \times U \times P,\R)$, and, $\forall (t, \xi, \zeta, \pi) \in [0,T] \times \Omega \times U \times P$, $D_{G,2}f^0(t,\xi, \zeta,\, \pi)$ exists. Moreover, for all $\pi\in P$, for all non-empty compact $K$ s.t. $K\subset \Omega\times U$, $\sup_{(t,\xi,\zeta)\in[0,T]\times K} \|D_{G,2}f^0(t,\xi,\zeta,\pi)\| <+\infty$.
\item[{\bf (IC2)}] For all $t\in [0,T]$, $D_{H,(2,3,4)}f^0(t,x[\pi_0](t),u[\pi_0](t),\pi_0)$ exists. Moreover, for all $\pi\in P$, $D_{H,3}f^0(T,x[\pi](T),u[\pi](T),\pi)$ exists, and the function \\$[ \pi \mapsto D_{H,3}f^0(T,x[\pi](T),u[\pi](T),\pi)] \in C^0(P,Y^*).$
\item[{\bf (IC3)}] There exists $\kappa \in \mathcal{L}^1(([0,T],\mathcal{B}([0,T])),\mathfrak{m}_1;\R_+)$, there exists $\rho>0$ s.t.,\\ $\forall t\in [0,T],$ $\forall (\xi_1,\zeta_1,\pi_1),\, (\xi_2,\zeta_2,\pi_2) \in B_{\|\cdot\|_1}((x[\pi_0](t),u[\pi_0](t),\pi_0),\rho),\\ 
|f^0(t,\xi_1,\zeta_1,\pi_1)-f^0(t,\xi_2,\zeta_2,\pi_2)|\le \kappa(t)\|(\xi_1,\zeta_1,\pi_1)-(\xi_2,\zeta_2,\pi_2)\|.$
\item[{\bf (IC4)}]
For all $(t, \zeta, \pi) \in [0,T] \times U \times P$, $D_{F,2}f^0(t,x[\pi](t),\zeta,\pi)$ exists and, $\forall \pi\in P,$ $[(t,\zeta)\mapsto D_{F,2}f^0(t,x[\pi](t),\zeta,\pi)]\in C^0([0,T]\times U,X^*)$,
\end{itemize}
where $\mathcal{B}([0,T])$ is the Borel tribe on $[0,T]$ and $\mathfrak{m}_1$ is the canonical Borel measure on $[0,T]$.\\
Notice that (IC1) concerns the continuity and the partial G\^ateaux differentiability; (IC2) concerns the Hadamard differentiability, (IC3) concerns a partial Lipschitz condition, and (IC4) concerns the partial Fr\'echet differentiability. if, $x\in C^0(P,C^0([0,T],\Omega))$ and, for all $t\in[0,T]$, $f^0(t,\cdot,\cdot,\cdot)$ is Fr\'echet differentiable on $\Omega\times U\times P$, and if $D_{F,(2,3,4)}f^0$ is continuous on $[0,T] \times \Omega\times U\times P$ then (IC1)-(IC4) are fulfilled. In our approach we want to weaken the conditions on $f^0$.   \\
{\bf Conditions on the vector field. }
\begin{itemize}
\item[{\bf (V1)}] For all $\pi\in P,\,[(t,\xi, \zeta) \mapsto f(t,\xi, \zeta,\pi)]\in C^0([0,T] \times \Omega \times U, X)$, and, for all $(t, \xi, \zeta, \pi) \in [0,T] \times \Omega \times U \times P$, $D_{G,2}f(t,\xi, \zeta,\, \pi)$ exists. Moreover, for all $\pi\in P$, for all non-empty compact $K$ s.t. $K\subset \Omega\times U$, $\sup_{(t,\xi,\zeta)\in [0,T]\times K} \|D_{G,2}f(t,\xi,\zeta,\pi)\|<+\infty$.
\item[{\bf (V2)}] For all $t\in [0,T]$, $D_{H,(2,3,4)}f(t,x[\pi_0](t),u[\pi_0](t),\pi_0)$ exist and, for all $\pi\in P$, $D_{H,3}f(T,x[\pi](T),u[\pi](T),\pi)$ exists.
\item[{\bf (V3)}] For all $(t, \zeta, \pi) \in [0,T] \times U \times P$, $D_{F,2}f(t,x[\pi](t),\zeta,\pi)$ exists and, $\forall \pi\in P,$ $[(t,\zeta)\mapsto D_{F,2}f(t,x[\pi](t),\zeta,\pi)]\in C^0([0,T]\times U,\mathcal{L}(X,X))$.
\end{itemize}
We can do a comment on (V1)-(V3) which is similar to the comment on (IC1)-(IC4) which is given just after (IC4). \\
{\bf Conditions on the terminal constraints functions and the terminal function of the criterion.}
\begin{itemize}
\item[{\bf (CT1)}] For all $\phi \in \{g^i : 0\le i \le m\}\cup \{h^j : 1\le j \le q\}$, $D_{H}\phi(x[\pi_0](T),\pi_0)$ exists and, $\forall \pi \in P,\, D_{H,1}\phi(x[\pi](T),\pi)$ exists. 
\item[{\bf (CT2)}] For all $\pi\in P$, for all $j\in\{1,...,q\}$, $h^j(\cdot,\pi)$ is continuous on a neighborhood of $x[\pi](T)$.
\end{itemize}
{\bf Conditions on the terminal constraints functions, the terminal function of the criterion and the vector field.}
\begin{itemize}
\item[{\bf (CVT1)}] $(D_{H,1}g^i(x[\pi_0](T),\pi_0)\circ D_{H,3}f(T,x[\pi_0](T),u[\pi_0](T),\pi_0)$,$\\$ $D_{H,1}h^j(x[\pi_0](T),\pi_0)\circ D_{H,3}f(T,x[\pi_0](T),u[\pi_0](T),\pi_0))_{1 \le i \le m, \,1\le j \le q}$ is linearly free.
\item[{\bf (CVT2)}] For all $\phi \in \{g^i : 0\le i \le m\}\cup \{h^j : 1\le j \le q\}$,\\ $[\pi\mapsto D_{H,1}\phi(x[\pi](T),\pi)\circ D_{H,3}f(T,x[\pi](T),u[\pi](T),\pi)]$ belongs to \\$ C^0(P,Y^*)$.
\end{itemize}
{\bf Conditions on the control space}
\begin{itemize}
\item[{\bf (ESP)}] There exists $(\cdot|\cdot)$ an inner product on $Y^*$ s.t. $(\cdot|\cdot) \in C^0((Y^*,\|\cdot\|_{Y^*})^2,\R)$.
\end{itemize}
\begin{theorem}\label{th61}
Under (SO), (IC1), (IC2), (IC3), (IC4), (V1), (V2), (V3), (CT1), (CT2), (CVT1), (CVT2) and (ESP), $D_G^+V[\pi_0;\delta\pi]$ exists and 
\\$D_G^+V[\pi_0;\delta\pi]=D_{H,2}g^0(x[\pi_0](T),\pi_0)\cdot \delta\pi + \sum_{i=1}^m \lambda_i[\pi_0]D_{H,2}g^i(x[\pi_0](T), \pi_0)\cdot \delta\pi\\
 +\sum_{j=1}^q \mu_j[\pi_0]D_{H,2}h^j(x[\pi_0](T), \pi_0)\cdot  \delta\pi\\
+\int_{[0,T]} D_{H,4}f^0(t,x[\pi_0](t),u[\pi_0](t),\pi_0)\cdot \delta\pi \, d\mathfrak{m}_1(t)\\
+\int_{[0,T]} p[\pi_0](t)\cdot D_{H,4}f(t,x[\pi_0](t),u[\pi_0](t),\pi_0)\cdot \delta\pi\,  d\mathfrak{m}_1(t),
$
where $(\lambda_i[\pi_0])_{0\le i\le m}$, with $\lambda_0[\pi_0]=1$, $(\mu_j[\pi_0])_{1\le j\le q}$ (respectively $p[\pi_0]$) are the unique respectively  multipliers (respectively the unique adjoint function) of the Pontryagin Theorem applied to the solution $(x[\pi_0],u[\pi_0])$ of $(\mathcal{B},\,\pi_0)$.
\end{theorem}
In order to provide a result on the G\^ateaux differentiability of $V$ at $\pi_0$, we introduce the following strengthened conditions.
\begin{itemize}
\item[{\bf (SO-bis)}] For all $\delta\pi \in Z$, $D_G^+x[\pi_0;\delta\pi]$ and $D_G^+u[\pi_0;\delta\pi]$ exists.
\item[{\bf (V4)}] There exists $\underline{c}\in  \mathcal{L}^1(([0,T],\mathcal{B}([0,T])),\mathfrak{m}_1;\R_+)$ s.t.  \\
$ \forall t\in [0,T],\, \|D_{H,4}f(t,x[\pi_0](t),u[\pi_0](t),\pi_0)\| \le \underline{c}(t).$
\end{itemize}
\begin{corollary}\label{cor62}
Under the assumptions of Theorem \ref{th61}, assuming in addition (SO-bis) and (V4), $V$ is G\^ateaux differentiable at $\pi_0$ and the formula of $D_GV[\pi_0]$ is given by the formula of Theorem \ref{th61}.
\end{corollary}
In order to provide a result on the continuously Fr\'echet differentiability of $V$, we introduce the following strengthened conditions
\begin{itemize}
\item[{\bf (SO-ter)}] The functions $[\pi\mapsto x[\pi]]$ and $[\pi\mapsto u[\pi]]$ are continuous at $\pi_0$ and, for all $\pi \in P$, for all $\delta\pi \in Z$, $D_G^+x[\pi;\delta\pi]$ and $D_G^+u[\pi;\delta\pi]$ exist.
\item[{\bf (IC5)}] For all $\pi\in P$,\, for all $t\in [0,T],$ $D_{H,(2,3,4)}f^0(t,x[\pi](t),u[\pi](t),\pi)$ and, for all $t\in [0,T],\, [\pi \mapsto D_{H,(2,4)}f^0(t,x[\pi](t),u[\pi](t),\pi)]\in C^0(P,(X\times Z)^*)$.
\item[{\bf (IC6)}]For all $\pi\in P$, $[t\mapsto D_{H,4}f^0(t,x[\pi](t),u[\pi](t),\pi)]$ belongs to \\$\mathcal{L}^0(([0,T],\mathcal{B}([0,T])),(Z^*,\mathcal{B}(Z^*)))$. 
\item[{\bf (V5)}] For all $\pi\in P$,\, for all $t\in [0,T]$, $D_{H,(2,3,4)}f(t,x[\pi](t),u[\pi](t),\pi)$ exist and for all $t\in[0,T]$, $[\pi \mapsto D_{H,(2,4)}f(t,x[\pi](t),u[\pi](t),\pi)]\in C^0(P,\mathcal{L}(X\times Z,X)).$
\item[{\bf (V6)}] For all $\pi\in P$, $[t\mapsto D_{H,4}f(t,x[\pi](t),u[\pi](t),\pi)]$ belongs to \\ $\mathcal{L}^0(([0,T],\mathcal{B}([0,T])),(\mathcal{L}(Z,X),\mathcal{B}(\mathcal{L}(Z,X)))$ and there exists \\$c\in  \mathcal{L}^1(([0,T],\mathcal{B}([0,T])),\mathfrak{m}_1;\R_+)$ s.t. $ \forall t\in [0,T],\, \forall \pi\in P,\,\\
 \|D_{H,(2,4)}f(t,x[\pi](t),u[\pi](t),\pi)\| \le c(t)$.
\item[{\bf (CT3)}] For all $\phi \in \{g^i : 0\le i \le m\}\cup \{h^j : 1\le j \le q\}$, for all $\pi\in P$, $\phi$ is Hadamard differentiable at $(x[\pi](T),\pi)$ and, $[\pi \mapsto D_H\phi(x[\pi](T),\pi)]\in C^0(P, (X\times Z)^*).$
\end{itemize}
$\mathcal{L}^0$ denotes the space of all measurable functions.
\begin{corollary}\label{cor63}
Under the assumptions of Corollary \ref{cor62}, if, in addition (SO-ter), (IC5), (IC6), (V5), (V6) and (CT3) are fulfilled, then $V$ is continuously Fr\'echet differentiable on $\mathfrak{W}$ which is an open neighborhood of $\pi_0$ and for all $\pi\in \mathfrak{W}$, for all $\delta \pi\in Z$,  
$
D_FV[\pi]\cdot \delta \pi=D_{H,2}g^0(x[\pi](T),\pi)\cdot \delta \pi + \sum_{i=1}^m \lambda_i[\pi]D_{H,2}g^i(x[\pi](T), \pi)\cdot \delta \pi\\
+\sum_{j=1}^q \mu_j[\pi]D_{H,2}h^j(x[\pi](T), \pi)\cdot \delta \pi\\
+\int_{[0,T]} D_{H,4}f^0(t,x[\pi](t),u[\pi](t),\pi) \cdot \delta \pi \, d\mathfrak{m}_1(t)\\
+\int_{[0,T]} p[\pi](t)\cdot D_{H,4}f(t,x[\pi](t),u[\pi](t),\pi) \cdot \delta \pi\, d\mathfrak{m}_1(t),
$
where $(\lambda_i[\pi])_{0\le i\le m}$, with \\
$\lambda_0[\pi]=1$, $(\mu_j[\pi])_{1\le j\le q}$ (respectively $p[\pi]$) are the unique respectively multipliers (respectively the unique adjoint function) of the Pontryagin Theorem applied to the solution $(x[\pi],u[\pi])$ of $(\mathcal{B},\,\pi)$.
\end{corollary}
\subsection{Proof of Theorem \ref{th61}}
We begin to establish a generalization of Lemma 5.2 in \cite{BY2}.   
\begin{lemma}\label{lem55} Let $E$ be a real normed vector space, $G$ be a non-empty open subset of $E$, ${\mathfrak f}:[0,T]\times G \rightarrow \R$ be a function and $\mathfrak{x}_0\in NPC^0_R([0,T],G )$. We consider the following conditions:
\begin{itemize}
\item[{\bf (i)}] ${\mathfrak f}\in  PCP^0([0,T]\times G, \R)$.
\item[{\bf (ii)}] There exists $\rho_1>0$ and there exists $\zeta\in \mathcal{L}^1(([0,T],\mathcal{B}([0,T])),\mathfrak{m}_1;\R_+)$ s.t., $\forall t \in [0,T], \forall {\mathfrak u}_1, {\mathfrak u}_2 \in B(\mathfrak{x}_0(t),\rho_1)$, $|{\mathfrak f}(t,{\mathfrak u}_1)-{\mathfrak f}(t,{\mathfrak u}_2)| \le \zeta(t)\|{\mathfrak u}_1-{\mathfrak u}_2\|.$
\item[{\bf (iii)}] For all $t\in [0,T]$, $D_{H,2}{\mathfrak f}(t,\mathfrak{x}_0(t))$ exists.
\end{itemize}
We consider the functional $F:NPC^0_R([0,T],G)\rightarrow \R$ defined by, \\
for all ${\mathfrak x}\in NPC^0_R([0,T],G)$, $F({\mathfrak x}):=\int_0^T {\mathfrak f}(t,{\mathfrak x}(t))\, dt$. The following assertions hold. 
\begin{itemize}
\item[{\bf (a)}] $NPC^0_R([0,T],G)$ is open in $NPC^0_R([0,T],E)$.  
\item[{\bf (b)}] Under (i)-(ii), $F$ is well defined and Lipschitzean on $B_{\|\cdot\|_{\infty}}(\mathfrak{x}_0,\rho_1)$.  
\item[{\bf (c)}] Under (i)-(iii), $F$ is Hadamard differentiable at $\mathfrak{x}_0$ and \\ for all $\mathfrak{h}\in NPC^0_R([0,T],E)$,\\ $[t\mapsto D_{H,2}{\mathfrak f}(t,\mathfrak{x}_0(t))\cdot \mathfrak{h}(t)]\in \mathcal{L}^1(([0,T],\mathcal{B}([0,T])),\mathfrak{m}_1;\R)$ and for all $\mathfrak{h}\in NPC^0_R([0,T],E),\; D_HF(\mathfrak{x}_0)\cdot \mathfrak{h}=\int_{[0,T]} D_{H,2}{\mathfrak f}(t,\mathfrak{x}_0(t))\cdot \mathfrak{h}(t)\; d\mathfrak{m}_1(t).$
\end{itemize}
\end{lemma}
\begin{proof}
{\bf (a)} Let $\mathfrak{x}\in NPC_R^0([0,T],G)$; we have the closure $cl(\mathfrak{x}([0,T])) \subset G$. If $(\tau_i)_{0\le i\le \mathfrak{n}+1}$ is the list of discontinuity points of $\mathfrak{x}$, when $i\in\{0,...,\mathfrak{n}\}$, we define $\mathfrak{x}_i:[\tau_i,\tau_{i+1}]\rightarrow G$ by setting $\mathfrak{x}_i(t):=\mathfrak{x}(t)$ if $t\in[\tau_i,\tau_{i+1}[$ and $\mathfrak{x}_i(\tau_{i+1}):=\mathfrak{x}(\tau_{i+1}-)$. Hence, we have $\mathfrak{x}_i\in C^0([\tau_i,\tau_{i+1}],G)$, and then $\mathfrak{x}_i([\tau_i,\tau_{i+1}])$ is compact and moreover we have $cl(\mathfrak{x}([0,T])):=\bigcup_{0\le i\le \mathfrak{n}} \mathfrak{x}_i([\tau_i,\tau_{i+1}])$ which is compact as a finite union of compacts.\\
Using the continuity of the function $[u\mapsto d(u,E\setminus G):=\inf\{ \|u-v\| :v\in E\setminus G \}$], the closedness of $E\setminus G$ and the Optimization theorem of Weierstrass setting $\alpha:=\inf\{ d(u,E\setminus G): u\in cl(\mathfrak{x}([0,T]))\}$, we have $\alpha>0$, and then we easily verify that $B_{\|\cdot\|_\infty}(\mathfrak{x},\frac{\alpha}{2})\subset NPC_R^0([0,T],G)$; and so (a) is proven.\\
{\bf (b)} When $\mathfrak{x}\in NPC_R^0([0,T])$, we see that $[t\mapsto \mathfrak{f}(t,\mathfrak{x}(t))]$ is regulated and consequently, it is Riemann integrable on $[0,T]$ cf. (\cite{Di}, p. 168) hence $F(\mathfrak{x})$ is well defined. (ii) implies that $F$ is Lipschitzean on $B_{\|\cdot\|_\infty}(\mathfrak{x},\rho_1)$; and so (b) is proven.\\
{\bf (c)} Let ${\mathfrak h}\in NPC^0_R([0,T],E)$, ${\mathfrak h} \neq 0$. We set $\theta^0:=\frac{1}{\|{\mathfrak h}\|_\infty}\min\{\rho,\frac{\alpha}{2}\}>0$. Let $(\theta_n)_{n\in \N} \in \, ]0,\theta^0[^{\N}$ s.t. $\lim\limits_{\substack{n \to +\infty}} \theta_n=0$.\\
Since the Hadamard differentiability implies the G\^ateaux differentiability, from (iii), for all $t\in[0,T]$, we have $D_{H,2}\mathfrak{f}(t,\mathfrak{x}_0(t))\cdot \mathfrak{h}(t)=D_{G,2}\mathfrak{f}(t,\mathfrak{x}_0(t))\cdot\mathfrak{h}(t)=\lim\limits_{\substack{n\to+\infty}} \Psi_n(t)$.\\
where $\Psi_n(t):=\frac{1}{\theta_n}(\mathfrak{f}(t,\mathfrak{x}_0(t)+\theta_n\mathfrak{h}(t))-\mathfrak{f}(t,\mathfrak{x}_0(t)))$.\\
Since $[t\mapsto \mathfrak{f}(t,\mathfrak{x}_0(t)+\theta_n\mathfrak{h}(t))]$ and $[t\mapsto \mathfrak{f}(t,\mathfrak{x}_0(t))]$ are regulated, they are uniform (therefore pointwise) limits of sequences of step functions, hence they are Borel functions, and then $\Psi_n$ is a Borel function as a pointwise limit of a sequence of Borel functions. Hence $[t\mapsto D_{H,2}\mathfrak{f}(t,\mathfrak{x}_0(t))\cdot \mathfrak{h}(t)]$ is a Borel functions.
Using (ii), we see that that $\|\Psi_n(t)\|\le \zeta(t)\|\mathfrak{h}\|_\infty$ for all $t\in[0,T]$ and for all $n\in \N$. Since $\zeta\|\mathfrak{h}\|_\infty \in \mathcal{L}^1(([0,T],\mathcal{B}([0,T])),\mathfrak{m}_1;\R_+)$ we can use the Dominated Convergence Theorem of Lebesgue to assert that $[t\mapsto D_{H,2}\mathfrak{f}(t,\mathfrak{x}_0(t))\cdot \mathfrak{h}(t)] \in \mathcal{L}^1(([0,T],\mathcal{B}([0,T])),\mathfrak{m}_1;\R)$ and also that \\
$\int_{[0,T]} D_{H,2}\mathfrak{f}(t,\mathfrak{x}_0(t))\cdot \mathfrak{h}(t)d\mathfrak{m}_1(t)=\lim\limits_{\substack{n\to +\infty}} \int_{[0,T]}\Psi_n(t)d\mathfrak{m}_1(t)=\lim\limits_{\substack{n\to +\infty}} \frac{1}{\theta_n}(F(\mathfrak{x}_0+\theta_n\mathfrak{h})-F(\mathfrak{x}_0))$; hence $D_G^+F(\mathfrak{x}_0;\mathfrak{h})$ exists and we have 
\begin{equation}\label{31}
D_G^+F(\mathfrak{x}_0;\mathfrak{h})=\int_{[0,T]} D_{H,2}\mathfrak{f}(t,\mathfrak{x}_0(t))\cdot \mathfrak{h}(t)d\mathfrak{m}_1(t).
\end{equation}
Using the linearity of the integral and of the Hadamard differential, we obtain that $D_G^+F(\mathfrak{x}_0;\cdot)$ is linear, and since $|D_G^+F(\mathfrak{x}_0,\mathfrak{h})|\le \|\zeta\|_{L^1}\|\mathfrak{h}\|_\infty$ for all \\$\mathfrak{h}\in NPC^0_R([0,T],E)$, we obtain that $D_G^+F(\mathfrak{x}_0;\cdot)$ is continuous and consequently $F$ is G\^ateaux differentiable at $\mathfrak{x}_0$. Under (a), since $F$ is Lipschitzean on a ball centered at $\mathfrak{x}_0$, using (\cite{F}, p.259), we obtain that $D_HF(\mathfrak{x}_0)$ exists and the formula of $D_HF(\mathfrak{x}_0)\cdot \mathfrak{h}$ is given by $D_G^+F(\mathfrak{x}_0;\mathfrak{h})$ and (\ref{31}). 
\end{proof}
\noindent
{\bf 1$^{st}$ step: existence of $D^+_G V(\pi_0)$.}
We consider $E:=X\times Y\times Z$, $G:=\Omega\times U\times P$ which is open in $E$, $\mathfrak{f}:[0,T]\times G\rightarrow \R$ the function defined by $\mathfrak{f}(t,(\xi,\zeta,\pi)):=f^0(t,\xi,\zeta,\pi)$, $\mathfrak{x}[\pi](t)=(x[\pi](t),u[\pi](t),\pi)$ and when 
$\mathfrak{x}\in NPC^0([0,T],G)$, we consider the function $F(\mathfrak{x}):=\int_0^T \mathfrak{f}(t,\mathfrak{x}(t))dt$ as in Lemma \ref{lem55}.
We want to use Lemma \ref{lem55} with $\mathfrak{x}_0=\mathfrak{x}[\pi_0]\in NPC^0([0,T],G)$. From (IC1), (IC2) and (IC3), the assumptions (i)-(iii) of Lemma \ref{lem55} are fulfilled, therefore we obtain that, for all $v\in NPC^0_d([0,T],X\times Y\times Z)$,
\begin{equation}\label{Dfopi0-0}
[t \mapsto D_{H,(2,3,4)}f^0(t,x[\pi_0](t),u[\pi_0](t),\pi_0)\cdot v(t)]  
\in \mathcal{L}^1(([0,T],\mathcal{B}([0,T])),\mathfrak{m}_1;\R),
 \end{equation}
and $F$ is Hadamard differentiable at $\mathfrak{x}[\pi_0]$. 
Next, since (SO), $D_G^+\mathfrak{x}[\pi_0; \delta\pi]$ exists and for all $t\in [0,T]$, $D_G^+\mathfrak{x}[\pi_0; \delta\pi](t):=(D_G^+x[\pi_0;\delta \pi](t),D_G^+u[\pi_0;\delta \pi](t),\delta\pi)$.\\
From (SO), we have also for all $\pi \in P$, $V[\pi]=F(\mathfrak{x}[\pi])+g^0(x[\pi](T),\pi).$
Therefore, using (\cite{F}, (4.2.5) p.263), we have
 $D_G^+V[\pi_0;\delta\pi]=D_HF(\mathfrak{x}[\pi_0])\cdot D_G^+\mathfrak{x}[\pi_0;\delta\pi]+D_Hg^0(x[\pi_0](T),\pi_0)\cdot (D_G^+x[\pi_0;\delta \pi](T),\delta\pi).$\\
\noindent
{\bf 2$^{st}$ step: a first formulation of $D^+_G V(\pi_0)$.}
Using the formula of Lemma \ref{lem55} for $F$ with $\mathfrak{x}_0=\mathfrak{x}[\pi_0]$ and the partial differentials in the previous formula, we obtain  
\begin{equation}\label{eq63}
\left.
\begin{array}{l}
D_G^+V[\pi_0;\delta\pi]\\
=\int_{[0,T]} D_{H,2}f^0(t,x[\pi_0](t),u[\pi_0](t),\pi_0)\cdot D_G^+x[\pi_0; \delta\pi](t)\, d\mathfrak{m}_1(t) \\
+\int_{[0,T]} D_{H,3}f^0(t,x[\pi_0](t),u[\pi_0](t),\pi_0)\cdot D_G^+u[\pi_0; \delta\pi](t)\, d\mathfrak{m}_1(t) \\
+\int_{[0,T]} D_{H,4}f^0(t,x[\pi_0](t),u[\pi_0](t),\pi_0)\cdot \delta\pi\, d\mathfrak{m}_1(t) \\
 + D_{H,1}g^0(x[\pi_0](T),\pi_0)\cdot D_G^+x[\pi_0;\delta\pi](T)+D_{H,2}g^0(x[\pi_0](T),\pi_0)\cdot\delta\pi.
\end{array}
\right\}
\end{equation}	
At this time, we see that the second and last terms of (\ref{eq63}) are present in the formula of Theorem \ref{th61}. We ought to transform the other terms.\\
\noindent
{\bf 3$^{st}$ step: the existence of multipliers and the adjoint function of the Pontryagin Theorem.} 
Thanks to (CVT1) and (CVT2), by using Lemma 4.1 in \cite{BY2}, there exists $Q$, an open neighborhood of $\pi_0$, $Q\subset P$, s.t.
\begin{equation}\label{eq62}
\left.
\begin{array}{l}
\forall \pi \in Q,\, ((D_{H,1}g^i(x[\pi](T),\pi)\circ D_{H,3}f(T,x[\pi](T),u[\pi](T),\pi))_{1 \le i \le m},\\ (D_{H,1}h^j(x[\pi](T),\pi)\circ D_{H,3}f(T,x[\pi](T),u[\pi](T),\pi))_{1\le j \le q})\\ \text{ is linearly free. }
\end{array}
\right\}
\end{equation}
Setting, $\forall (\pi,t,\xi,\zeta)\in Q\times [0,T]\times \Omega\times U$, $f_{\pi}^{0}(t,\xi,\zeta)=f^0(t,\xi,\zeta,\pi)$, $f_{\pi}(t,\xi,\zeta)=f(t,\xi,\zeta,\pi)$, $g_\pi^i(\xi):=g^i(\xi,\pi)$ ($0\le i\le m$), $h_\pi^j(\xi):=h^j(\xi,\pi) (1\le j\le q)$, we see that $({\mathcal B},\pi)$ is a special case of the problem of (${\mathcal B}$) of Section 2 of this paper.   
For all $\pi \in Q$, note that (SO) implies that $(x[\pi],u[\pi])$ is a solution of $(\mathcal{B},\pi)$,
 (IC1) and (IC4) imply that (A{\sc i}1) is fulfilled, (IC1) implies that (A{\sc i}2) is fulfilled, (V1) and (V3) imply that (A{\sc v}1) is fulfilled, (V1) implies that (A{\sc v}2) is fulfilled, (CT1) implies that (A{\sc t}1) is fulfilled, (CT2) implies that (A{\sc t}2) is fulfilled, (CT2) implies that (A{\sc t}2) is fulfilled. Hence, the assumptions of Theorem \ref{th21} are fulfilled. Moreover, note that (\ref{eq62}) and (V2) implies that (LI) is fulfilled, and using (IC2), the assumptions of Corollary \ref{cor13}, (iv),  hold. \\
Hence, we obtain the existence and the uniqueness of $((\lambda_i[\pi])_{0 \le i \le m}$, $(\mu_j[\pi])_{1 \le j \le q},\\ p[\pi])\in \R^{k+1}\times \R^{l}\times PC^1([0,T],X^*)$ with $\lambda_0[\pi]=1$ s.t. the following conditions are fulfilled. 
\begin{itemize} 
\item[{\sf (Si)}] For all $\pi\in Q$, for all $i \in \{0,...,m \}$, $\lambda_{i}[\pi] \geq 0$.
\item[{\sf (S${\ell}$)}] For all $\pi\in Q$, for all $i \in \{1,...,m \}$, $\lambda_i[\pi] g^i(x[\pi](T),\pi) = 0$.
\item[{\sf (TC)}] For all $\pi\in Q$,\\ $D_{H,1}g^0(x[\pi](T),\pi) + \sum_{i = 1}^{m} \lambda_i[\pi] D_{H,1}g^i(x[\pi](T),\pi) \\ + \sum_{j = 1}^q \mu_j[\pi] D_{H,1}h^j(x[\pi](T),\pi) = p[\pi](T)$.
\item[{\sf (AE)}] For all $\pi\in Q$, for all $t \in [0,T]$, ${\underline d}p[\pi](t) \\= -p[\pi](t)\circ D_{F,2}f(t,x[\pi](t),u[\pi](t),\pi)-D_{F,2}f^0(t,x[\pi](t),u[\pi](t),\pi).$  
\item[{\sf (MP)}] For all $\pi\in Q$, for all $t \in [0,T]$, for all $\zeta \in U$, \\
$H_{\pi}(t, x[\pi](t), u[\pi](t), p[\pi](t), 1) \geq H_{\pi}(t, x[\pi](t), \zeta, p[\pi](t), 1).$     
\end{itemize}
\noindent
{\bf $4^{th}$ step: the transformation of the partial differentials of $f^0$ with respect to the state variable and the control variable.} 
For all $t\in[0,T]$, for all real normed space $\mathfrak{Y}$, we consider the evaluation operator $ev_t^{\mathfrak Y}: NPC^0_R([0,T],{\mathfrak Y}) \rightarrow \mathfrak{Y}$ defined by $ev_t^\mathfrak{Y}(\varphi):=\varphi(t)$ when $\varphi\in NPC_R^0([0,T],\mathfrak{Y})$. Note that $ev_t^\mathfrak{Y}\in \mathcal{L}(NPC^0([0,T],\mathfrak{Y}),\mathfrak{Y})$. We can rewrite the evolution equation of $({\mathcal B},\pi)$ in the following form, \\
$\forall \pi\in P,\, \forall t\in[0,T]$, $(ev_t^X\circ \underline{d} \circ x)[\pi]=f(t,\cdot,\cdot,\cdot) \circ (ev_t^X(x)[\cdot],ev_t^Y(u)[\cdot],id_Z)[\pi]$.  
Using (\cite{F}, p.253), we obtain 
$D_G^+(ev_t^X\circ \underline{d} \circ x)[\pi_0;\delta\pi]=D_H(ev_t^X \circ \underline{d})(x([\pi_0]) \cdot D_G^+x[\pi_0;\delta\pi] \\
=(ev_t^X \circ \underline{d}) \cdot D_G^+x[\pi_0;\delta\pi]=ev_t^X(\underline{d}(D_G^+x[\pi_0,\delta\pi])\\
=\underline{d}(D_G^+x[\pi_0;\delta\pi])(t)$ i.e. we have the following inversion of the two notions of differentiation $\underline{d}$ and $D_G^+$. 
\begin{equation}
D_G^+(\underline{d}x(t))[\pi_0;\delta\pi]=\underline{d}(D_G^+x[\pi_0;\delta\pi])(t).
\end{equation}
From (V2), we also have $D_G^+(f(t,\cdot,\cdot,\cdot) \circ (ev_t^X(x)[\cdot],ev_t^Y(u)[\cdot],id_Z))[\pi_0;\delta\pi]\\
=D_{H,(2,3,4)}f(t,x[\pi_0](t),u[\pi_0](t),\pi_0)\cdot (D_G^+x[\pi_0;\delta\pi](t),D_G^+u[\pi_0;\delta\pi](t),\delta\pi)$.
Therefore, we have, for all $ t \in [0,T]$, ${\underline d}[D_G^+x[\pi_0;\delta\pi](t)\\
= D_{H,(2,3,4)}f(t,x[\pi_0](t),u[\pi_0](t),\pi_0)\cdot (D_G^+x[\pi_0;\delta\pi](t),D_G^+u[\pi_0;\delta\pi](t),\delta\pi).$\\
From {\sf (MP)}, (V2) and (IC2), we have, for all $t\in[0,T]$,\\
$p[\pi_0](t)\circ D_{H,3}f(t,x[\pi_0](t),u[\pi_0](t),\pi_0)+D_{H,3}f^0(t,x[\pi_0](t),u[\pi_0](t),\pi_0)=0.$\\
Consequently, for all $t\in [0,T]$, we have
\[\begin{array}{ll}
D_{H,(2,3)}f^0(t,x[\pi_0](t),u[\pi_0](t),\pi_0)\cdot (D_G^+x[\pi_0;\delta\pi](t),D_G^+u[\pi_0;\delta\pi](t))\\
=-{\underline d}[p[\pi_0]](t)\cdot D_G^+x[\pi_0;\delta\pi](t)-\\p[\pi_0](t)\cdot D_{H,2}f(t,x[\pi_0](t),u[\pi_0](t),\pi_0)\cdot D_G^+x[\pi_0;\delta\pi](t)-\\ p[\pi_0](t)\cdot D_{H,3}f(t,x[\pi_0](t),u[\pi_0](t),\pi_0)\cdot D_G^+u[\pi_0;\delta\pi](t) \\
= -{\underline d}[p[\pi_0]](t)\cdot D_G^+x[\pi_0;\delta\pi](t)-p[\pi_0)(t)\cdot D_{H,(2,3,4)}f(t,x[\pi_0](t),u[\pi_0](t),\pi_0)\cdot \\(D_G^+x[\pi_0;\delta\pi](t),D_G^+u[\pi_0;\delta\pi](t),\delta\pi)\\
+p[\pi_0](t)\cdot D_{H,4}f(t,x[\pi_0](t),u[\pi_0](t),\pi_0)\cdot \delta\pi\\
=-{\underline d}[p[\pi_0]](t)\cdot D_G^+x[\pi_0;\delta\pi](t)-p[\pi_0](t)\cdot {\underline d}[D_G^+x[\pi_0;\delta\pi]](t)+\\ p[\pi_0](t)\cdot D_{H,4}f(t,x[\pi_0](t),u[\pi_0](t),\pi_0)\cdot \delta\pi. 
\end{array}
\]
Therefore, we have
\begin{equation}\label{eq66}
\left.
\begin{array}{l}
D_G^+V[\pi_0;\delta\pi]=\\
\int_{[0,T]} -{\underline d}[p[\pi_0]](t)\cdot D_G^+x[\pi_0;\delta\pi](t)-p[\pi_0](t)\cdot \underline{d}[D_G^+x[\pi_0;\delta\pi]](t)\, d\mathfrak{m}_1(t)\\+\int_{[0,T]} p[\pi_0](t)\cdot D_{H,4}f(t,x[\pi_0](t),u[\pi_0](t),\pi_0)\cdot \delta\pi \, d\mathfrak{m}_1(t)\\
+ \int_{[0,T]} D_{H,4}f^0(t,x[\pi_0](t),u[\pi_0](t),\pi_0)\cdot \delta\pi \, d\mathfrak{m}_1(t)\\
+D_{H,1}g_0(x[\pi_0](T),\pi_0)\cdot D_G^+x[\pi_0;\delta\pi](T)+D_{H,2}g_0(x[\pi_0](T),\pi_0)\cdot \delta\pi.\\
\end{array}
\right\}
\end{equation}
We consider the function $\psi :[0,T] \rightarrow \R $ defined by, for all $t\in [0,T]$, $\psi(t)=p[\pi_0](t) \cdot D_G^+x[\pi_0;\delta\pi](t)$.\\  
Since $\psi \in PC^1([0,T],X^*)$ and $D_G^+x[\pi_0;\delta\pi]\in PC^1([0,T],X)$, we have\\
 $\psi\in PC^1([0,T],\R)$ and, for all $t\in [0,T]$, we have ${\underline d}p[\pi_0](t)\cdot D_G^+x[\pi_0;\delta\pi](t)+p[\pi_0](t)\cdot {\underline d}[D_G^+x[\pi_0;\delta\pi]](t)$.
i.e. ${\underline d}\psi(t)= {\underline d}p[\pi_0](t)\cdot D_G^+x[\pi_0;\delta\pi](t)+p[\pi_0](t)\cdot {\underline d}[D_G^+x[\pi_0;\delta\pi]](t)$.
Since $\psi\in PC^1([0,T],\R)$, we have, for all $t\in [0,T]$, $\psi(T)-\psi(0)=\int_{0}^T {\underline d}\psi(t)dt$.\\
Moreover, we have $D_G^+x[\pi_0;\delta\pi](0)=0.$ \\Therefore $ -\int_{0}^T {\underline d}\psi(t)dt= -\psi(T)+\psi(0)=-p[\pi_0](T)\cdot D_G^+x[\pi_0;\delta\pi](T).$        \\
Therefore, we obtain 
\begin{equation}\label{eq66ter}
\left.
\begin{array}{l}
D_G^+V[\pi_0;\delta\pi] =-p[\pi_0](T)\cdot D_G^+x[\pi_0;\delta\pi](T)\\
+\int_{[0,T]} p[\pi_0](t)\cdot D_{H,4}f(t,x[\pi_0](t),u[\pi_0](t),\pi_0)\cdot \delta\pi \, d\mathfrak{m}_1(t)\\
+ \int_{[0,T]} D_{H,4}f^0(t,x[\pi_0](t),u[\pi_0](t),\pi_0)\cdot \delta\pi \, d\mathfrak{m}_1(t)\\
+D_{H,1}g^0(x[\pi_0](T),\pi_0)\cdot D_G^+x[\pi_0;\delta\pi](T)\\
+D_{H,2}g^0(x[\pi_0](T),\pi_0)\cdot \delta\pi.
\end{array}
\right\}
\end{equation}
\noindent
{\bf $5^{th}$ step: the transformation of the first and the second to last terms of (\ref{eq66ter}).}
From {\sf (TC)}, we obtain 
\begin{equation}\label{eq66terter}
\left.
\begin{array}{l}
D_G^+V[\pi_0;\delta\pi] =-\sum_{i=1}^m \lambda_i[\pi_0]D_{H,1}g^i(x[\pi_0](T),\pi_0)\cdot D_G^+x[\pi_0;\delta\pi](T)\\
-\sum_{j=1}^q \mu_j[\pi_0]D_{H,1}h^j(x[\pi_0](T),\pi_0)\cdot D_G^+x[\pi_0;\delta\pi](T)\\
+\int_{[0,T]} p[\pi_0](t)\cdot D_{H,4}f(t,x[\pi_0](t),u[\pi_0](t),\pi_0)\cdot \delta\pi \, d\mathfrak{m}_1(t)\\
+ \int_{[0,T]} D_{H,4}f^0(t,x[\pi_0](t),u[\pi_0](t),\pi_0)\cdot \delta\pi \, d\mathfrak{m}_1(t)\\
+D_{H,2}g^0(x[\pi_0](T),\pi_0)\cdot \delta\pi.
\end{array}
\right\}
\end{equation}
Besides, by using {\sf (MP)} and (V2), we have, for all $\pi \in Q,\\
p[\pi](T)\circ D_{H,3}f(T,x[\pi](T),u[\pi](T),\pi)+D_{H,3}f^0(T,x[\pi](T),u[\pi](T),\pi)=0$.
Consequently, by using {\sf (TC)}, we have
\begin{equation}\label{eq67}
\left.
\begin{array}{l}
\forall \pi\in Q,\, \sum_{i = 1}^{m} \lambda_i[\pi] D_{H,1}g^i(x[\pi](T), \pi)\circ D_{H,3}f(T,x[\pi](T),u[\pi](T),\pi)\\ + \sum_{j = 1}^q \mu_j[\pi] D_{H,1}h^j(x[\pi](T),\pi)\circ D_{H,3}f(T,x[\pi](T),u[\pi](T),\pi)\\
=-D_{H,1}g^0(x[\pi](T), \pi) \circ D_{H,3}f(T,x[\pi](T),u[\pi](T),\pi)\\-D_{H,3}f^0(T,x[\pi](T),u[\pi](T),\pi).
\end{array}
\right\}
\end{equation}
In the following lemma, we etablish the continuity of the multipliers with respect to $\pi$. 
\begin{lemma}\label{lem61}
For all $i\in\{1,...,\,m\},\,\lambda_i \in C^0(Q,\R_+)$, and, for all $j\in\{1,...,\, q\},$ $\mu_j\in C^0(Q,\R)$.
\end{lemma}
\begin{proof}
First, for all $\pi \in Q$, we set $F_\pi := span \{ e_i[\pi] : 1 \leq i \leq m+q \}$ and $F:=\bigcup_{\pi\in Q} (F_\pi \times \{\pi\})$ where for all $i\in\{1,...,m\},$ $e_i[\pi]=D_{H,1}g^i(x[\pi](T), \pi)\circ D_{H,3}f(T,x[\pi](T),u[\pi](T),\pi)$ and for all $j\in\{1,...,q\},\\ e_{m+j}[\pi]=D_{H,1}h^j(x[\pi](T), \pi) \circ D_{H,3}f(T,x[\pi](T),u[\pi](T),\pi)$.
For all $({\bf x},\pi)\in F$, we denote by ${\bf x}_\alpha(x,\pi)$ the $\alpha$-th coordinate of ${\bf x}$ in the basis $(e_j[\pi])_{1 \leq j \leq m+q}$.\\ 
From (IC2), (V2), (CVT2), (ESP) and (\ref{eq62}), by using Lemma 4.3 in \cite{BY2}, with $\E=Y^*, \W=Z, W=Q,$ we obtain that, $\forall \alpha\in \{1,...,m+q\}$, ${\bf x}_\alpha$ is continuous on $\bigcup_{\pi\in Q} (F_\pi \times \{\pi\})$.\\
Consequently, since (\ref{eq67}), we have, for all $\pi\in Q$, \\${\bf x}[\pi]:=-D_{H,1}g^0(x[\pi](T),\pi)\circ  D_{H,3}f(T,x[\pi](T),u[\pi](T),\pi)-\\D_{H,3}f^0(T,x[\pi](T),u[\pi](T),\pi)\in F_\pi$.\\
Hence, we have for all $i\in\{1,...,m\}$,\,  $\lambda_i={\bf x}_i\circ ({\bf x}, id_Q)\in C^0(Q,\R)$, and \\
for all $j\in\{1,...,q\}$, $\mu_j={\bf x}_{m+j}\circ ({\bf x}, id_Q)\in C^0(Q,\R).$
\end{proof}
Let $i\in\{1,...,\,m\}$; if $\lambda_i[\pi_0]>0$, using Lemma \ref{lem61}, there exists a neighborhood $N$ of $\pi_0$ in $Q$ s.t., for all $\pi \in N$, $\lambda_i[\pi]>0.$
Consequently, by using {\sf (S${\ell}$)}, we obtain that, for all $\pi \in N,\; g^i(x[\pi](T),\,\pi)=0.$
From (SO) and (CT1), we have $D_{H,1}g^i(x[\pi_0](T), \pi_0)\cdot D_G^+x[\pi_0;\delta\pi](T)+D_{H,2}g^i(x[\pi_0](T), \pi_0)\cdot \delta\pi=0.$
Hence we have $\lambda_i[\pi_0]D_{H,2}g^i(x[\pi_0](T), \pi_0)\cdot \delta\pi  = -\lambda_i[\pi_0]D_{H,1}g^i(x[\pi_0](T), \pi_0)\cdot D_G^+x[\pi_0;\delta\pi](T).$

Moreover, if $\lambda_i[\pi_0]=0$, then we also have
$\lambda_i[\pi_0]D_{H,2}g^i(x[\pi_0](T), \pi_0)\cdot \delta\pi \\ = -\lambda_i[\pi_0]D_{H,1}g^i(x[\pi_0](T), \pi_0)\cdot D_G^+x[\pi_0;\delta\pi](T).
$
Hence, we obtain
\begin{equation}\label{eq610}
\left.
\begin{array}{l}
\forall i\in\{1,...,\,m\},\, \lambda_i[\pi_0]D_{H,2}g^i(x[\pi_0](T), \pi_0)\cdot \delta\pi \\ = -\lambda_i[\pi_0]D_{H,1}g^i(x[\pi_0](T), \pi_0)\cdot D_G^+x[\pi_0;\delta\pi](T).
\end{array}
\right\}
\end{equation}
Let $j\in\{1,...,\,q\}$; remark that for all $\pi\in Q$, $h^j(x[\pi](T),\pi)=0.$\\
From (SO) and (CT1), we have\\
$D_{H,1}h^j(x[\pi_0](T), \pi_0)\cdot D_G^+x[\pi_0;\delta\pi](T)+D_{H,2}h^j(x[\pi_0](T), \pi_0)\cdot \delta\pi=0.$
Consequently, we obtain 
\begin{equation}\label{eq611}
\left.
\begin{array}{l}
\mu_j[\pi_0]D_{H,2}h^j(x[\pi_0](T), \pi_0)\cdot \delta\pi  = -\mu_j[\pi_0]D_{H,1}h^j(x[\pi_0](T), \pi_0)\cdot D_G^+x[\pi_0;\delta\pi](T).
\end{array}
\right.
\end{equation}
From (\ref{eq66terter}), (\ref{eq610}) and (\ref{eq611}), we have
\[
\begin{array}{l}
D_G^+V[\pi_0;\delta\pi]=D_{H,2}g^0(x[\pi_0](T), \pi_0)\cdot \delta\pi+\sum_{i=1}^m \lambda_i[\pi_0]D_{H,2}g^i(x[\pi_0](T), \pi_0)\cdot \delta\pi \\+\sum_{j=1}^q  \mu_j[\pi_0]D_{H,2}h^j(x[\pi_0](T), \pi_0)\cdot \delta\pi \\
+\int_{[0,T]} D_{H,4}f^0(t,x[\pi_0](t),u[\pi_0](t),\pi_0)\cdot \delta\pi \, d\mathfrak{m}_1(t)\\
+\int_{[0,T]} p[\pi_0](t)\cdot D_{H,4}f(t,x[\pi_0](t),u[\pi_0](t),\pi_0)\cdot \delta\pi \, d\mathfrak{m}_1(t).
\end{array}
\]
\subsection{Proof of Corollary \ref{cor62}}
For all $\delta\pi\in Z$, from (SO-bis), $D_G^+x[\pi_0;\delta\pi]$ and $D_G^+u[\pi_0;\delta\pi]$ exist.\\
Consequently, the assumptions of Theorem \ref{th61} are fulfilled for every direction $\delta\pi\in Z$.\\
Therefore, using Theorem \ref{th61}, we have 
\begin{equation}\label{Vg-0}
\left.
\begin{array}{l}
\forall \delta\pi\in Z,\\
D_G^+V[\pi_0;\delta\pi]=D_{H,2}g^0(x[\pi_0](T), \pi_0)\cdot \delta\pi\\
+\sum_{i=1}^m \lambda_i[\pi_0]D_{H,2}g^i(x[\pi_0](T), \pi_0)\cdot \delta\pi \\+\sum_{j=1}^q  \mu_j[\pi_0]D_{H,2}h^j(x[\pi_0](T), \pi_0)\cdot \delta\pi \\
+\int_{[0,T]} D_{H,4}f^0(t,x[\pi_0](t),u[\pi_0](t),\pi_0)\cdot \delta\pi \, d\mathfrak{m}_1(t)\\
+\int_{[0,T]} p(\pi_0)(t)\cdot D_{H,4}f(t,x[\pi_0](t),u[\pi_0](t),\pi_0)\cdot \delta\pi \, d\mathfrak{m}_1(t).
\end{array}
\right\}
\end{equation}
Moreover, we have 
\begin{equation}\label{fokap}
\forall t\in [0,T], \; \|D_{H,(2,3,4)}f^0(t,x[\pi_0](t),u[\pi_0](t),\pi_0)\| \le \kappa(t).
\end{equation}
Since (\ref{fokap}), using the linearity property of the Borel integral, we have \\
$[\delta\pi \mapsto\int_{[0,T]} D_{H,4}f^0(t,x[\pi_0](t),u[\pi_0](t),\pi_0)\cdot \delta\pi \, d\mathfrak{m}_1(t)]\in Z^*.$\\
Besides, from (V4), we have for all $t\in [0,T]$, \\
$\|p[\pi_0](t)\circ D_{H,4}f(t,x[\pi_0](t),u[\pi_0](t),\pi_0)\|\\
\le \|p[\pi_0](t)\| \|D_{H,4}f(t,x[\pi_0](t),u[\pi_0](t),\pi_0)\|\le \|p[\pi_0]\|_{\infty}\underline{c}(t).\\
$
Consequently, using the linearity property of the Borel integral, we have \\
$[\delta\pi \mapsto \int_{[0,T]} p[\pi_0](t)\cdot D_{H,4}f(t,x[\pi_0](t),u[\pi_0](t),\pi_0)\cdot \delta\pi \, d\mathfrak{m}_1(t)]\in Z^*.
$
Therefore, we have $[\delta\pi\mapsto D_G^+V[\pi_0;\delta\pi]]\in Z^*$. Therefore $V$ is G\^ateaux-differentiable at $\pi_0$.\\
\subsection{Proof of Corollary \ref{cor63}}
Note that, by using (SO-ter), the function $\mathfrak{x}: P\rightarrow NPC^0_d([0,T],\Omega\times U\times P )$, defined by, for all $\pi\in P$, for all $t\in [0,T]$, $\mathfrak{x}[\pi](t)=(x[\pi](t),u[\pi](t),\pi)$ is continuous at $\pi_0$.\\
Therefore, we have
\begin{equation}\label{Vf-1}
\exists \overline{r}>0\,\text{ s.t. }\forall \pi\in B(\pi_0,\overline{r}), \, \|\mathfrak{x}[\pi]-\mathfrak{x}[\pi_0]\|_\infty <\frac{\rho}{2}.
\end{equation}  
We set $\mathfrak{a}:=\min\{\overline{r},\frac{\rho}{2}\}$.
Note that, we also have 
\begin{equation}\label{lemVf1}
\forall \pi \in B(\pi_0,\mathfrak{a}),\, \forall t\in [0,T], \, B_{\|\cdot\|_1}(\mathfrak{x}[\pi](t),\mathfrak{a}) \subset B_{\|\cdot\|_1}(\mathfrak{x}[\pi_0](t),\rho).
\end{equation}
From (IC3), by using (\ref{lemVf1}), we have
\begin{equation}\label{Vf-2} 
\left.
\begin{array}{l}
\forall \pi\in B(\pi_0,\mathfrak{a}), \forall t\in [0,T],\,  \forall (\xi_1,\zeta_1,\pi_1),\, (\xi_2,\zeta_2,\pi_2) \in B_{\|\cdot\|_1}(z(\pi)(t),\mathfrak{a}), \\|f^0(t,\xi_1,\zeta_1,\pi_1)-f^0(t,\xi_2,\zeta_2,\pi_2)|\le \kappa(t)\|(\xi_1,\zeta_1,\pi_1)-(\xi_2,\zeta_2,\pi_2)\|_1.   
\end{array}
\right\}
\end{equation}
We set $\mathfrak{W}:=B(\pi_0,\mathfrak{a})\cap Q\subset P$. Let $\pi\in \mathfrak{W}$. Note that ${\mathfrak W}$ is an open neighborhood of $\pi$ in $Q$ and, for all $\overline{\pi}\in {\mathfrak W}$, $(x[\overline{\pi}],u[\overline{\pi}])$ is a solution of $(\mathcal{B},\overline{\pi})$, consequently (SO) is fulfilled $\pi_0=\pi$.\\
$\pi_0$ is not present in the conditions (IC1), (IC4), (V1), (V3), (CT2) and (CVT2), that is why using our assumptions this conditions are already verified. \\
Using our additional assumptions, the conditions (SO-bis), (IC2), (IC3), (V2), (V4) and (CT1) are fulfilled with $\pi_0=\pi$.\\
Besides, the assertions (\ref{Vf-2}) implies that (IC3) is fulfilled with $\pi_0=\pi$, and the assertion (\ref{eq62}) imples that (CVT1) with $\pi_0=\pi$.
Hence, for all $\pi\in \mathfrak{W}$, the assumptions of Corollary \ref{cor62} are fulfilled with $\pi_0=\pi$
Therefore, we can use the Theorem \ref{th61} and we obtain that $V$ is G\^ateaux differentiable for every $\pi\in \mathfrak{W}$ and  
\begin{equation}\label{Vf-3} 
\left.
\begin{array}{l}
\forall \pi \in \mathfrak{W}, \forall \delta\pi\in Z  \\
 D_GV[\pi]\cdot \delta\pi=D_{H,2}g^0(x[\pi](T),\pi)\cdot \delta\pi + \sum_{i=1}^m \lambda_i[\pi]D_{H,2}g^i(x[\pi](T), \pi)\cdot \delta\pi\\
+\sum_{j=1}^q \mu_j(\pi)D_{H,2}h^j(x[\pi](T), \pi)\cdot \delta\pi\\
+\int_{[0,T]} D_{H,4}f^0(t,x[\pi](t),u[\pi](t),\pi)\cdot \delta\pi \, d\mathfrak{m}_1(t)\\
+\int_{[0,T]} p[\pi](t)\cdot D_{H,4}f(t,x[\pi](t),u[\pi](t),\pi)\cdot \delta\pi\, d\mathfrak{m}_1(t).
\end{array}
\right\}
\end{equation}
Besides, by using (\ref{Vf-2}) and (IC5), we have 
\begin{equation}\label{lemVf2}
\forall \pi\in B(\pi_0,\mathfrak{a}),\; \forall t\in [0,T], \, \|D_{H,(2,3,4)} f^0(t,x[\pi](t),u[\pi](t),\pi)\|\le \kappa(t).
\end{equation}
For all $i\in \{2,4\}$, for all $t\in [0,T]$,  $\pi\in {\mathfrak W}$, we set\\
$\mathfrak{f}_i(t,\pi)=D_{H,i} f(t,x[\pi](t),u[\pi](t),\pi)\text{ and }\mathfrak{f}^0_i(t,\pi)=D_{H,i} f^0(t,x[\pi](t),u[\pi](t),\pi).$
In the following lemma, we prove that the adjoint function are continous with respect to the parameter $\pi$.
\begin{lemma}\label{lemVf3}
$[\pi\mapsto p[\pi]]\in C^0({\mathfrak W},(PC^1([0,T],X^*),\|\cdot\|_{\infty}))$.
\end{lemma}
\begin{proof}  
From (TC), (CT3) and Lemma \ref{lem61}, we have
\begin{equation}\label{eq615bis}
\left.
\begin{array}{l}
[\pi \mapsto p[\pi](T)]\in C^0(Q,X^*).
\end{array}
\right.
\end{equation}
Let $\hat{\pi}\in \mathfrak{W}$. We consider the functions $\varphi_1:[0,T]\times \mathfrak{W}\rightarrow \R$, defined by $(t,\pi)\in [0,T]\times \mathfrak{W}$, $\varphi_1(t,\pi):=\|{\mathfrak f}_2(t,\pi)-{\mathfrak f}_2(t,\hat{\pi})\|$, and $\varphi^0_1:[0,T]\times \mathfrak{W}\rightarrow \R$, defined by, for all $(t,\pi)\in [0,T]\times \mathfrak{W}$, $\varphi^0_1(t,\pi):=\|\mathfrak{f}^0_2(t,\pi)-\mathfrak{f}^0_2(t,\hat{\pi})\|$. Note that using (IC4) and Lemma \ref{lem113}, we have for all $\pi\in \mathfrak{W},$ $\varphi^0_1(\cdot,\pi)\in NPC_d^0([0,T],\R)$. Besides, using (\ref{lemVf2}), we have for all $\pi\in \mathfrak{W}$, for all $t\in [0,T]$,\,  $\varphi^0_1(t,\pi)\le 2\kappa(t)$. Next, using (IC5), we have, for all $t\in [0,T]$, $\lim\limits_{\substack{\pi\to \hat{\pi}}} \varphi^0_1(t,\pi)=\varphi^0_1(t,\hat{\pi})=0$.\\
Therefore, using the Dominated Convergence Theorem of Lebesgue, the functional $\psi^0_1: \mathfrak{W}\rightarrow \R$, defined by, for all $\pi\in \mathfrak{W}$, $\psi^0_1(\pi):=\int_{0}^{T} \varphi_1^0(t,\pi)\, dt$, is continuous at $\hat{\pi}$ i.e.
\begin{equation}\label{pc-4}
 \lim\limits_{\substack{\pi\to \hat{\pi}}} \psi^0_1(\pi)=\psi^0_1(\hat{\pi})=0.
\end{equation}
Using (V3) and Lemma \ref{lem113}, we have, for all $\pi\in \mathfrak{W},\, \varphi_1(\cdot,\pi)\in NPC_d^0([0,T],\R)$.\\
From (V6) we have ,$\forall \pi\in \mathfrak{W},\, \forall t\in [0,T],\,  \varphi_1(t,\pi)\le 2c(t)$.\\
Besides, using (V5), we have, for all $t\in [0,T]$, $\lim\limits_{\substack{\pi\to \hat{\pi}}} \varphi_1(t,\pi)=\varphi_1(t,\hat{\pi})=0.$\\
Consequently, using the Dominated Convergence Theorem of Lebesgue, the functional $\psi_1: \mathfrak{W}\rightarrow \R$, defined by, for all $\pi\in \mathfrak{W}$, $\psi_1(\pi):=\int_{0}^{T} \varphi_1(t,\pi)\, dt$, is continuous at $\hat{\pi}$ i.e.
\begin{equation}\label{pc-8}
 \lim\limits_{\substack{\pi\to \hat{\pi}}} \psi_1(\pi)=\psi_1(\hat{\pi})=0.
\end{equation} 
For all $\pi\in \mathfrak{W}$, for all $t\in [0,T]$, we have the following inequalities: 
\[
\begin{array}{l}
\|p[\pi](t)-p[\hat{\pi}](t)\|=\|p[\pi](T)+\int_{T}^{t}[-p[\pi](s)\circ {\mathfrak f}_2(s,\pi)-{\mathfrak f}_2^0(s,\pi)]\, ds-\\(p[\hat{\pi}](T)+\int_{T}^{t}[-p[\hat{\pi}](s)\circ {\mathfrak f}_2(s,\hat{\pi})-{\mathfrak f}_2^0(s,\hat{\pi})]ds)\| \\
\le \|p[\pi](T)-p[\hat{\pi}](T)\| + \int_t^T \|p[\pi](s)\circ {\mathfrak f}_2(s,\pi)-p[\hat{\pi}](s)\circ {\mathfrak f}_2(s,\hat{\pi})\|\,ds \\+
\int_t^T \|{\mathfrak f}_2^0(s,\pi)-{\mathfrak f}_2^0(s,\hat{\pi})\|ds\\
\le \|p[\pi](T)-p[\hat{\pi}](T)\| + \int_t^T \|p[\pi](s)\circ {\mathfrak f}_2(s,\pi)-p[\hat{\pi}](s)\circ {\mathfrak f}_2(s,\pi)+\\p[\hat{\pi}](s)\circ {\mathfrak f}_2(s,\pi)-p[\hat{\pi}](s)\circ {\mathfrak f}_2(s,\hat{\pi})\|\,ds+\int_t^T \|{\mathfrak f}_2^0(s,\pi)-{\mathfrak f}_2^0(s,\hat{\pi})\|ds\\
\le  \|p[\pi](T)-p[\hat{\pi}](T)\| + \int_t^T \|p[\pi](s)-p[\hat{\pi}](s)\|\|{\mathfrak f}_2(s,\pi)\|\, ds \\+\int_t^T\|p[\hat{\pi}](s)\| \|{\mathfrak f}_2(s,\pi)- {\mathfrak f}_2(s,\hat{\pi})\|\,ds+\psi_1^0(\pi)\\ 
\le  \|p[\pi](T)-p[\hat{\pi}](T)\| + \int_{[t,T]} \|p[\pi](s)-p[\hat{\pi}](s)\|\,c(s)\, d\mathfrak{m}_1(s) +\|p[\hat{\pi}]\|_\infty \psi_1(\pi)\\+\psi_1^0(\pi).\\
\end{array}
\]
Since $[s \mapsto \|p[\pi](s)-p[\hat{\pi}](s)\|]\in C^0([0,T],\R)$ and $c\in \mathcal{L}^1(([0,T],\mathcal{B}([0,T])),\mathfrak{m}_1;\R_+)$, by using the lemma of Gronwall (\cite{ATF}, p.183), we have, for all $t\in [0,T]$,\\
$\|p[\pi](t)-p[\hat{\pi}](t)\| \le [ \|p[\pi](T)-p[\hat{\pi}](T)\|+\|p[\hat{\pi}]\|_\infty \psi_1(\pi)+ \\ \psi_1^0(\pi)] \exp(\int_{[t,T]} c(s)\, d\mathfrak{m}_1(s))$. \\ 
Therefore, we have
\begin{equation}\label{pc-9}
 \|p[\pi]-p[\hat{\pi}]\|_\infty \le [ \|p[\pi](T)-p[\hat{\pi}](T)\|+\|p[\hat{\pi}]\|_\infty \psi_1(\pi)+  \psi_1^0(\pi)] \exp(\int_{[0,T]} c(s)\, d\mathfrak{m}_1(s)).
\end{equation}
Hence, using (\ref{eq615bis}), (\ref{pc-4}) and (\ref{pc-8}), we have $ \lim\limits_{\substack{\pi\to \hat{\pi}}}  \|p[\pi]-p[\hat{\pi}]\|_\infty=0.$\\
Consequently, we have $[\pi\mapsto p[\pi]]\in C^0(\mathfrak{W},(PC^1([0,T],X^*),\|\cdot\|_{\infty}))$.
\end{proof}
Next, we consider the function $\Psi_2: {\mathfrak W} \rightarrow Z^*$, defined by, for all $\pi\in {\mathfrak W}$, for all $\delta\pi\in Z$, 
$$\Psi_2(\pi)\cdot \delta\pi:=\int_{[0,T]} D_{H,4}f^0(t,x[\pi](t),u[\pi](t),\pi)\cdot \delta\pi\, d\mathfrak{m}_1(t).$$
\begin{lemma}\label{lemVf4}
$\Psi_2\in C^0({\mathfrak W},Z^*)$.
\end{lemma} 
\begin{proof}
Let $\hat{\pi}\in {\mathfrak W}$. We consider the fonction $\varphi^0_2:[0,T]\times {\mathfrak W}\rightarrow \R$ defined by $\varphi^0_2(t,\pi):=\|{\mathfrak f}^0_4(t,\pi)-{\mathfrak f}^0_4(t,\hat{\pi})\|$.\\
Using (IC6), we have, for all $\pi\in {\mathfrak W}$, $\varphi^0_2(\cdot,\pi)\in \mathcal{L}^0(([0,T],\mathcal{B}([0,T])),(\R,\mathcal{B}(\R)))$.
Next, using (\ref{lemVf2}), we have, $\forall \pi\in {\mathfrak W}$, $\forall t\in [0,T]$,  $\varphi^0_2(t,\pi)\le 2\kappa(t)$.
Besides, from (IC5), we have, for all $t\in [0,T]$,\,  $\lim\limits_{\substack{\pi\to \hat{\pi}}} \varphi^0_2(t,\pi)=\varphi^0_2(t,\hat{\pi})=0.$
Hence, using the Dominated Convergence Theorem of Lebesgue, the functional $\psi^0_2: {\mathfrak W}\rightarrow \R$, defined by, $\psi^0_2(\pi):=\int_{[0,T]} \varphi_2^0(t,\pi)\, d\mathfrak{m}_1(t)$ is continuous at $\hat{\pi}$ i.e.
\begin{equation}\label{Vf4-4}
\lim\limits_{\substack{\pi\to \hat{\pi}}} \psi^0_2(\pi)=\psi^0_2(\hat{\pi})=0.
\end{equation}
For all $\pi\in {\mathfrak W}$, we have $\|\Psi_2(\pi)-\Psi_2(\hat{\pi})\| \le \psi_2^0(\pi)$.\\
Conseqently, using (\ref{Vf4-4}), we have $\lim\limits_{\substack{\pi\to \hat{\pi}}} \|\Psi_2(\pi)-\Psi_2(\hat{\pi})\|=0.$
Hence, we have proven the lemma.
\end{proof}
Now, we consider, the function $\Psi_3: {\mathfrak W} \rightarrow Z^*$, defined by, for all $\pi\in {\mathfrak W}$, for all $\delta\pi\in Z$, 
$\Psi_3(\pi)\cdot \delta\pi:=\int_{[0,T]} p[\pi](t)\cdot D_{H,4}f(t,x[\pi](t),u[\pi](t),\pi)\cdot \delta\pi\, d\mathfrak{m}_1(t).$
\begin{lemma}\label{lemVf5}
$\Psi_3\in C^0({\mathfrak W},Z^*)$.
\end{lemma} 
\begin{proof}
Let $\hat{\pi}\in {\mathfrak W}$. We consider the function $\varphi_2:[0,T]\times {\mathfrak W}\rightarrow \R$, defined by $\varphi_2(t,\pi):=\|{\mathfrak f}_4(t,\pi)-{\mathfrak f}_4(t,\hat{\pi})\|$.\\
From (V6), we have, for all $\pi\in {\mathfrak W}$, $\varphi_2(\cdot,\pi)\in \mathcal{L}^0(([0,T],\mathcal{B}([0,T])),(\R,\mathcal{B}(\R)))$.
Moreover, using (V6), we also have, $\forall \pi\in {\mathfrak W},\, \forall t\in [0,T],\,  \varphi_2(t,\pi)\le 2c(t)$.
Besides, from (V5), we have, for all $t\in [0,T]$, $\lim\limits_{\substack{\pi\to \hat{\pi}}} \varphi_2(t,\pi)=\varphi_2(t,\hat{\pi})=0$.
Hence, using the Dominated Convergence Theorem of Lebesgue the functional $\psi_2: {\mathfrak W}\rightarrow \R$, defined by $\pi\in {\mathfrak W}$, $\psi_2(\pi):=\int_{[0,T]} \varphi_2(t,\pi)\, d\mathfrak{m}_1(t)$, is continuous at $\hat{\pi}$ i.e.
\begin{equation}\label{Vf5-4}
\lim\limits_{\substack{\pi\to \hat{\pi}}} \psi_2(\pi)=\psi_2(\hat{\pi})=0.
\end{equation}
For all $\pi\in {\mathfrak W}$, we have
$
\|p[\pi](t)\circ {\mathfrak f}_4(t,\pi)-p[\hat{\pi}](t)\circ {\mathfrak f}_4(t,\hat{\pi}) \|\\
=\|p[\pi](t)\circ {\mathfrak f}_4(t,\pi)-p[\hat{\pi}](t)\circ {\mathfrak f}_4(t,\pi)+p[\hat{\pi}](t)\circ {\mathfrak f}_4(t,\pi)-p[\hat{\pi}](t)\circ {\mathfrak f}_4(t,\hat{\pi}) \| \\
\le  \|p[\pi]-p[\hat{\pi}]\|_\infty c(t)+\|p[\hat{\pi}]\|_\infty \varphi_2(t,\pi).\\
$
Consequently, we obtain 
$$
\|\Psi_3(\pi)-\Psi_3(\hat{\pi})\|\\
\le \|p[\pi]-p[\hat{\pi}]\|_\infty \int_{[0,T]} c(t)\,d\mathfrak{m}_1(t)+\|p[\hat{\pi}]\|_\infty \psi_2(\pi).
$$
Consequently, using Lemma \ref{lemVf3} and (\ref{Vf5-4}), we obtain $\lim\limits_{\substack{\pi \to\hat{\pi}}} \|\Psi_3(\pi)-\Psi_3(\hat{\pi})\|=0.$
Therefore, we have proven this lemma.
\end{proof}
From (\ref{Vf-3}), remark that 
$$
\left.
\begin{array}{ll}
\forall \pi \in {\mathfrak W}, & \null  \\
 D_GV[\pi]&=D_{H,2}g_0(x[\pi](T),\pi) + \sum_{i=1}^m \lambda_i[\pi]D_{H,2}g^i(x[\pi](T), \pi)\\
\null &+\sum_{j=1}^q \mu_j[\pi]D_{H,2}h^j(x[\pi](T), \pi)+ \Psi_2(\pi)+\Psi_3(\pi).
\end{array}
\right\}
$$
Consequently, using (CT3) and Lemmas \ref{lem61}, \ref{lemVf4} and \ref{lemVf5}, we obtain that, $D_GV\in C^0({\mathfrak W},Z^*).$   
Therefore, using Corollary 2, p. 144 in \cite{ATF}, we obtain that $V$ is Fr\'echet differentiable on ${\mathfrak W}$ and $D_FV[\pi]=D_GV[\pi]$ for all $\pi\in {\mathfrak W}$, and therefore $D_FV\in C^0({\mathfrak W},Z^*)$.


\begin{thebibliography}{00}
%
\bibitem{ATF} V.M. Alexeev, V.M. Tikhomirov, and S.V. Fomine, {\sf Commande optimale}, french edition, MIR, Moscow, 1982.
%
\bibitem{BL} J. Blot, {\it On the multiplier rules}, Optimization, {\bf 65}, 947-955, 2016.
%
\bibitem{BY} J. Blot, and H. Yilmaz, {\it A generalization of Michel's result on the Pontryagin Maximum Principle}, J. Optim. Theory  Appl., {\bf 183}, 792-812, 2019.
%
\bibitem{BY2} J. Blot, and H. Yilmaz, {\it Envelope theorems for static optimization and Calculus of Variations}, J. Math. Anal. Appl., {\bf 509}(Issue 2), doi.org/10.1016/j.jmaa.2021.125966, 2022.
%
\bibitem{Di} J. Dieudonn\'e, {\sf \'El\'ements d'analyse, tome 1: fondements de l'analyse moderne}, french edition, Gauthier-Villars, Paris, 1969.
%
\bibitem{F} T.M. Flett,  {\sf Differential analysis}, Cambridge University Press, Cambridge (UK), 1980.
%
\bibitem{GD} A. Granas, and J. Dugundji, {\sf Fixed point theory}, Springer-Verlag New York Inc., New-York, 2003.
%
\bibitem{PH} P. Hartman, {\sf Ordinary differential equations}, second edition, SIAM, Boston, 1982.
%
\bibitem{LA} S. Lang, {\sf Real and functional analysis}, third edition, Springer-Verlang New York, Inc., New-York, 1993.
%
\bibitem{PMP} P. Michel, {\it Une d\'emonstration \'el\'ementaire du principe du maximum de Pontryaguine}, Bulletin de Math\'ematiques \'Economiques, $N^o$14, Mars 1977.
%
\bibitem{Sc} L. Schwartz, {\sf Cours d'analyse; tome 1}, Hermann, Paris, 1967.
%
\bibitem{HY} H. Yilmaz, {\it A generalization of multiplier rules for infinite-dimensional optimization problems}, Optimization, {\bf 70}(8), 1825-1835, 2021.
\end{thebibliography}
\end{document}